\documentclass{article}
\usepackage[a4paper]{geometry}
\usepackage{latexsym}
\usepackage{enumerate}
\usepackage{amsthm,amsmath,amsfonts,amssymb}
\usepackage{graphicx,graphics}
\usepackage{xcolor}
\usepackage{epsfig,bm,epsf,float}
\usepackage{url}
\usepackage{hyperref}

\usepackage{minitoc}
\newcommand{\li}{{\textup{li}}}
\newcommand{\Li}{{\textup{Li}}}

\newcommand{\ra}{\rightarrow}

\theoremstyle{plain}
\newtheorem{thm}{Theorem}[section]
\newtheorem{cor}[thm]{Corollary}
\newtheorem{lem}[thm]{Lemma}
\newtheorem{prop}[thm]{Proposition}

\theoremstyle{definition}
\newtheorem{defn}[thm]{Definition}

\theoremstyle{remark}
\newtheorem{rem}[thm]{Remark}

\numberwithin{equation}{section}
\title{Sharper estimates for Chebyshev's functions $\vartheta$ and $\psi$}
\author{Sadegh Nazardonyavi, Semyon Yakubovich\\
\emph{Departamento de Matem\'{a}tica, Faculdade de Ci\^{e}ncias},\\ \emph{Universidade do Porto, 4169-007 Porto, Portugal}}
\begin{document}
\date{}
\maketitle
\begin{center}
Abstract
\end{center}
In this article we present some improved results for Chebyshev's functions $\vartheta$ and $\psi$ using the new zero-free region obtained by H. Kadiri and the calculated the first $10^{13}$ zeros of the Riemann zeta function on the critical line by Xavier Gourdon. The methods in the proofs are similar to those of Rosser-Shoenfeld papers on this subject.
\section{Chebyshev's functions}

\begin{defn}
For $x>0$ we define Chebyshev's $\psi$-function by the formula
$$
\psi(x)=\sum_{n\leq x}\Lambda(n),
$$
where
$$
\Lambda(n)=\left\{
  \begin{array}{lll}
    \log p&, & n=p^m\ \ \hbox{for some}\ m; \\\\
    0&, & \hbox{otherwise.}
  \end{array}
\right.
$$
\end{defn}

Since $\Lambda(n)=0$ unless $n$ is a prime power, we can write the definition of $\psi(x)$ as follows:
\begin{equation}\label{psi}
\psi(x)=\sum_{n\leq x}\Lambda(n)=\sum_{m=1}^\infty\sum_{p^m\leq x}\Lambda(p^m)=\sum_{m=1}^\infty\sum_{p\leq x^{1/m}} \log p.
\end{equation}
The sum on $m$ is actually a finite sum. In fact, the sum on $p$ is empty if $x^{1/m}<2$, that is, if
$(1/m)\log x<\log2$, or if
$$
m>\frac{\log x}{\log2}=\log_2x.
$$
Therefore, we have
$$
\psi(x)=\sum_{m\leq\log_2x}\sum_{p\leq x^{1/m}} \log p=\sum_{p\leq x}\left\lfloor\frac{\log x}{\log p}\right\rfloor\log p.
$$
This can be written in a slightly different form by introducing another function of Chebyshev.
\begin{defn}
If $x>0$, we define Chebyshev's $\vartheta$-function by the equation
$$
\vartheta(x)=\sum_{p\leq x} \log p.
$$
\end{defn}
The last formula for $\psi(x)$ can now be restated as follows:
\begin{equation}\label{psi sum of theta}
    \psi(x)=\sum_{m\leq\log_2x}\vartheta(x^{1/m}).
\end{equation}
Using M\"{o}bius inversion formula

\begin{align*}
\sum_{k\geq1}\mu(k)\psi(x^{1/k})=&\sum_{k\geq1}\mu(k)\sum_{l\geq1}\vartheta(x^{1/kl})=\sum_{n\geq1}\sum_{k|n}\mu(k)\vartheta(x^{1/n})\\
                                =&\sum_{n\geq1}\delta_{1,n}\vartheta(x^{1/n})=\vartheta(x),
\end{align*}
where
$$
\mu(n)=\left\{
  \begin{array}{lll}
    1&, & n=1; \\
    (-1)^k&, & n=\text{product of $k$ distinct primes};\\
    0&, &\hbox{otherwise,}
  \end{array}
\right.
$$
and $\delta_{i,j}$ is Kronecker's delta function
$$
\delta_{i,j}=\left\{
  \begin{array}{lll}
    1&, & i=j; \\
    0&, &i\neq j.
  \end{array}
\right.
$$
\begin{thm}[\cite{Apostol}, p. 76]
For $x>0$ we have
$$
0\leq\frac{\psi(x)}{x}-\frac{\vartheta(x)}{x}\leq\frac{1}{2\log2}\frac{\log^2x}{\sqrt{x}}.
$$
\end{thm}
Note that this inequality implies that
$$
\lim_{x\rightarrow\infty}\left(\frac{\psi(x)}{x}-\frac{\vartheta(x)}{x}\right)=0.
$$
In other words, if one of $\psi(x)/x$ or $\vartheta(x)/x$ tends to a limit then so does the other, and the two limits are equal.
\subsection{Relations connecting $\vartheta(x)$ and $\pi(x)$}
In 1896 J. Hadamard and C. J. de la Vall\'{e}e Poussin independently and almost simultaneously succeeded in proving that
$$
\lim_{x\ra\infty}\frac{\pi(x)\log x}{x}=1.
$$
This remarkable result is called the prime number theorem, and its proof was one of the crowning achievements of analytic number theory.

In this section we give two formulas relating $\vartheta(x)$ and $\pi(x)$. These can be used to show that the prime number theorem is equivalent to the limit relation
$$
\lim_{x\rightarrow\infty}\frac{\vartheta(x)}{x}=1.
$$
\begin{thm}[Abel's identity]
For any arithmetical function $a(n)$ let
$$
A(x)=\sum_{n\leq x} a(n),
$$
where $A(x)=0$ if $x<1$. Assume $f$ has a continuous derivative on the interval $[y, x]$, where $0<y<x$. Then we have
\begin{equation}\label{Abel identity y<x}
    \sum_{y<n\leq x} a(n)f(n)=A(x)f(x)-A(y)f(y)-\int_y^x A(t)f'(t)dt.
\end{equation}
\end{thm}
Now we use (\ref{Abel identity y<x}) to express $\vartheta(x)$ and $\pi(x)$ in terms of integrals.
\begin{thm}[\cite{Apostol}, p. 78]
For $x\geq2$ we have
\begin{equation}\label{theta by pi}
    \vartheta(x)=\pi(x)\log x-\int_2^x\frac{\pi(t)}{t}dt,
\end{equation}
\begin{equation}\label{pi by theta}
    \pi(x)=\frac{\vartheta(x)}{\log x}+\int_2^x\frac{\vartheta(t)}{t\log^2t}dt.
\end{equation}
\end{thm}
%
\subsection{Relations connecting $\psi(x)$ and $\Pi(x)$}
From Euler's identity
$$
\zeta(s)=\sum_{n=1}^\infty\frac{1}{n^s}=\prod_{p=2}^\infty\frac{1}{1-1/p^s},\qquad(\Re s>1).
$$
Taking logarithm
$$
\log\zeta(s)=-\sum_p\log\left(1-\frac{1}{p^s}\right)=\sum_{p,m}\frac{1}{mp^{ms}}.
$$
By differentiation,
$$
-\frac{\zeta'(s)}{\zeta(s)}=\sum_{p,m}\frac{\log p}{p^{ms}},
$$
or
$$
-\frac{\zeta'(s)}{\zeta(s)}=\sum_{n=1}^\infty\frac{\Lambda(n)}{n^s}.
$$
On the other hand,
$$
\psi(x)=\sum_{n\leq x}\Lambda(n).
$$
So, by Abel's identity
$$
-\frac{\zeta'(s)}{\zeta(s)}=s\int_1^\infty\frac{\psi(x)}{x^{s+1}}dx,\qquad(\Re s>1).
$$
The function
$$
\Pi(x)=\sum_{p^m\leq x}\frac1m=\pi(x)+\frac12\pi(x^{1/2})+\frac13\pi(x^{1/3})+\ldots.
$$
So
$$
\log\zeta(s)=\sum_{p,m}\frac{1/m}{p^{ms}}=\int_1^\infty\frac{d\Pi(x)}{x^{s}}dx=s\int_1^\infty\frac{\Pi(x)}{x^{s+1}}dx,\qquad(\Re s>1).
$$
Now the connection between $\Pi$ and $\psi$.
\begin{align*}
\Pi(x)=\sum_{p^m\leq x}\frac1m=&\sum_{2\leq n\leq x}\frac{\Lambda(n)}{\log n}\\
                              =&\frac{\psi(x)}{\log x}+\int_2^x\frac{\psi(t)}{t\log^2t}dt,
\end{align*}
Recall that
$$
\li(x)=\lim_{\varepsilon\ra0}\left\{\int_0^{1-\varepsilon}\frac{dt}{\log t}+\int_{1+\varepsilon}^x\frac{dt}{\log t}\right\}.
$$
Also
$$
\li(x)=\li(2)+\int_2^x\frac{dt}{\log t}=\li(2)+\frac{x}{\log x}-\frac{2}{\log 2}+\int_2^x\frac{t}{t\log^2t}dt,
$$
$$
\li(2)\approx1.04516,\qquad \frac{2}{\log 2}\approx2.88539.
$$
So
\begin{equation}\label{Pi(x)-li(x)}
\Pi(x)-\li(x)=\frac{\psi(x)-x}{\log x}+\int_2^x\frac{\psi(t)-t}{t\log^2t}dt+O(1).
\end{equation}
\begin{thm}[\cite{Apostol}, p. 79]
The following relations are logically equivalent:
\begin{equation}\label{pi logx/x=1}
    \lim_{x\rightarrow\infty}\frac{\pi(x)\log x}{x}=1,
\end{equation}
\begin{equation}\label{theta x/x=1}
    \lim_{x\rightarrow\infty}\frac{\vartheta(x)}{x}=1,
\end{equation}
\begin{equation}\label{psi x/x=1}
    \lim_{x\rightarrow\infty}\frac{\psi(x)}{x}=1.
\end{equation}
\end{thm}
\subsection{Chebyshev's functions and the Riemann zeta function}
The zeta function was introduced in mathematics as an analytic tool for studying prime numbers. Therefore, it is only natural that some of the most important applications of the zeta function belong to prime number theory. Here we shall be concerned with some of the most important of these applications.

Many problems in prime number theory may be formulated in terms of the functions $\pi$, $\vartheta$, and $\psi$.

\begin{prop}
We have
\begin{equation}\label{pi and psi inequality}
\frac{\psi(x)}{\log x}<\pi(x)<\int_2^x\frac{d\psi(t)}{\log t}.
\end{equation}
\end{prop}
\begin{proof}
Since
\begin{align}\label{pi and psi relation}
\pi(x)=&\sum_{p\leq x}1=\sum_{p\leq x}\frac{1}{\lfloor\log x/\log p\rfloor\log p}\left\lfloor\frac{\log x}{\log p}\right\rfloor\log p\nonumber\\
=&\int_{2}^x\frac{1}{\lfloor\log x/\log t\rfloor\log t}d\psi(t).
\end{align}
and
$$
1\leq \lfloor\frac{\log x}{\log t}\rfloor\leq \frac{\log x}{\log t}
$$
\end{proof}

Already Riemann, whose work was in many aspects decades beyond that of his contemporaries, stated the elegant formula, which says that the weighted function $\psi$ is in a certain sense more natural than $\pi$ and $\vartheta$, since it possesses a (relatively simple) explicit expression, and relates the order of $\psi(x)-x$ to a certain sum over non-trivial zeros of the zeta function; namely
\begin{equation}\label{12.6}
    \boxed{\psi(x)=x-\sum_\rho\frac{x^\rho}{\rho}-\frac{\zeta'(0)}{\zeta(0)}-\frac12\log(1-\frac{1}{x^2}),\quad(x>1,\ x\neq p^m),}
\end{equation}
where $\rho=\beta+i\gamma$ is a non-trivial zero of $\zeta(s)$, and
$$
\sum_\rho\frac{x^\rho}{\rho}=\lim_{T\rightarrow\infty}\sum_{|\gamma|\leq T}\frac{x^\rho}{\rho}.
$$
and when $x=p^m$, then in the left-hand side of (\ref{12.6}) put $\psi(x)-\frac12\Lambda(x)$.
This explicit expression for $\psi(x)$ was proved by H. von Mangoldt in 1895.\\

\subsection{The error term in the prime number theorem}
The size of the error term in the prime number theorem depends on the location of zeros of the Riemann zeta function \cite{McCurley}.
If
$$
\{s=\sigma+it:\ \sigma>1-\frac{c}{\log|t|},\ |t|>t_0\}
$$
is a zero free region, then an explicit error term in the prime number theorem is
\begin{thm}[\cite{Ellison}, p. 141]
There exists a constant $a>0$ such that for $x$ tending to infinity, we have
\begin{align*}
\psi(x)-x=&O(x\exp(-a\sqrt{\log x}))\\
\vartheta(x)-x=&O(x\exp(-a\sqrt{\log x}))\\
\Pi(x)-\Li(x)=&O(x\exp(-a\sqrt{\log x}))\\
\pi(x)-\Li(x)=&O(x\exp(-a\sqrt{\log x}))
\end{align*}
One can choose $a=1/15$ and all constants ``$O$'' are effective.
\end{thm}

\begin{thm}[\cite{Ellison}, p. 425]
There exist a positive constant $\alpha$ such that for $x$ infinity we have
$$
\pi(x)-\Li(x)=O\left\{x\exp(-\alpha\frac{(\log x)^{3/5}}{(\log\log x)^{1/5}})\right\}
$$
The constant ``$O$'' is effective and can take $\alpha=0.009$. The corresponding asymptotic formulas take place for $\vartheta(x)$, $\psi(x)$ and $\Pi(x)$.
\end{thm}
Cheng \cite{Cheng} gives an explicit zero-free region for the Riemann zeta-function derived from the Vinogradov- Korobov method.
%
He proves that the Riemann zeta-function does not vanish in the region
$$
\sigma\geq 1-\frac{0.00105}{\log^{2/3}|t|(\log\log|t|)^{1/3}},\qquad |t|\geq3
$$
In turn, he showes using these results that for all $x>10$
$$
|\pi(x)-\li(x)|\leq 11.88x(\log x)^{3/5}\exp(-\frac{1}{57}(\log x)^{3/5}(\log\log x)^{1/5})
$$
and for $x\geq e^{e^{44.08}}$, there is a prime between $x^3$ and $(x+1)^3$.
\begin{thm}[\cite{Ford}]
If
$$
|\zeta(\sigma+it)|\leq A|t|^{B(1-\sigma)^{3/2}}\log^{2/3}|t|,\qquad(\frac12\leq\sigma\leq1,\ |t|\geq3,\ A=76.2)
$$
holds with a certain constant $B$, then for large $|t|$, $\zeta(\sigma+it)\neq0$ for
$$
\sigma\geq1-\frac{0.05507B^{-2/3}}{(\log|t|)^{2/3}(\log\log|t|)^{1/3}}
$$
Taking $B=4.45$ gives the zero-free region.
\end{thm}

\subsection{The results of Ingham}
\begin{thm}[\cite{Ingham}, p. 66]
When $x\rightarrow\infty$,
\begin{equation}\label{psi th 24 ingham}
\psi(x)=x+O(x\exp(-a\sqrt{\log x\log\log x}))
\end{equation}
\begin{equation}\label{pi th 24 ingham}
\pi(x)=\li(x)+O(x\exp(-a\sqrt{\log x\log\log x}))
\end{equation}
where $a$ is a positive absolute constant.
\end{thm}
Let $\Theta$ be the upper bound of the real parts of the zeros of $\zeta(s)$. Clearly $\Theta\leq1$, since there is no zeros in $\sigma>1$. And from the existence of the non-trivial zeros $\rho$ and their symmetry about the line $\sigma=\frac12$ we infer that $\Theta\geq\frac12$. Thus $\frac12\leq\Theta\leq1$, and this is the most that is known about $\Theta$; but $\Theta =\frac12$ if (and only if) the Riemann hypothesis is true. We now have the following theorem, which is worthless if $\Theta=1$.
\begin{thm}[\cite{Ingham}, p. 83]
$$
\psi(x)=x+O(x^{\Theta}\log^2x)
$$
$$
\pi(x)=\li(x)+O(x^{\Theta}\log x)
$$
\end{thm}
\begin{thm}[\cite{Ingham}, p. 90]
If $\delta$ is any fixed positive number, then
$$
\psi(x)-x=\Omega_\pm(x^{\Theta-\delta})
$$
$$
\Pi(x)-\li(x)=\Omega_\pm(x^{\Theta-\delta})
$$
\end{thm}
\begin{thm}[\cite{Ingham}, p. 100]
We have
$$
\psi(x)-x=\Omega_\pm(x^{1/2}\log\log\log x)
$$
when $x\rightarrow\infty$. In fact,
$$
\limsup\frac{\psi(x)-x}{x^{1/2}\log\log\log x}\geq\frac12
$$
$$
\liminf\frac{\psi(x)-x}{x^{1/2}\log\log\log x}\leq-\frac12
$$
\end{thm}
\section{New Explicit Bounds for Some Functions of Prime Numbers}
Riemann Hypothesis verified until the $10^{13}$-th zero by Gourdon (October 12th 2004) \cite{Gourdon}.

Recall $N(T)$, $F(T)$ and $R(T)$ be defined as
\begin{align}
  N(T)=&\#\{\rho:\ \zeta(\rho)=0,\ 0<\gamma\leq T\}\label{N(T).}\\\nonumber\\
  F(T)=&\frac{T}{2\pi}\log\frac{T}{2\pi}-\frac{T}{2\pi}+\frac78\label{F(T)} \\\nonumber\\
  R(T)=&0.137\log T+0.443\log\log T+1.588\label{R(T)}\\\nonumber
\end{align}

\begin{thm}[\cite{Rosser-1941}]\label{N(T)-F(T) R(T)}
For $T\geq2$,
$$
|N(T)-F(T)|<R(T)
$$
\end{thm}

Choose $A$ such that $F(A)=10^{13}$. Then
\begin{align*}
A=&2,445,999,556,030.342,362,641\\
\log A=&28.525,474,972.
\end{align*}
\begin{lem}[\cite{Rosser-1939}]\label{lemma 8 of Rosser}
We have
$$
\sum_\rho\frac{1}{\gamma^{2}}<0.0463.
$$
\end{lem}

\begin{prop}\label{sum 1/gamma2,3,4,5,6,7}

\begin{align*}
\sum_\rho\frac{1}{|\gamma^{3}|}<0.00146435,&&\sum_\rho\frac{1}{\gamma^{4}}<7.43617\cdot10^{-5},
\end{align*}
\begin{align*}
\sum_\rho\frac{1}{|\gamma^{5}|}<4.46243\cdot10^{-6},&&\sum_\rho\frac{1}{\gamma^{6}}<2.88348\cdot10^{-7},
\end{align*}
\begin{align*}
\sum_\rho\frac{1}{|\gamma^{7}|}<1.93507\cdot10^{-8}.&&
\end{align*}
\end{prop}
\begin{proof}
We use the same method as in \cite{Rosser-1939}, and letting $r=29$ instead of 8.
\end{proof}


\begin{thm}[\cite{Kadiri}]
The Riemann zeta-function $\zeta(s)$ doesn't vanish in the region
$$
\sigma\geq1-\frac{1}{R_0\log|t|},\qquad(|t|\geq2,\ R_0=5.69693)
$$
\end{thm}
In other words, if $\rho=\beta+i\gamma$ is a zero of Riemann zeta function, then
$$
\beta<1-\frac{1}{R_0\log|t|},\qquad(|t|\geq2,\ R_0=5.69693)
$$

\subsection{Estimates for certain integrals to the Bessel functions}
Let
$$
K_\nu(z,x)=\frac12\int_x^\infty t^{\nu-1}H^z(t)dt
$$
where $z>0$, $x\geq0$ and
$$
H^z(t)=\exp\{-\frac12z(t+\frac1t)\}
$$
\begin{lem}[\cite{Rosser-1975}]\label{upper bound of K1(z,x), K2(z,x)}
$$
K_\nu(z,x)+K_{-\nu}(z,\frac1x)=K_\nu(z,0)=K_\nu(z)
$$

\begin{equation}\label{(2.30) R-S}
K_1(z,x)<\frac{e^{-z}}{2z}\left\{\left(1+\frac{3\sqrt{2}y}{8}\right)e^{-zy^2}+(\frac38+z)\sqrt{2}\int_y^\infty e^{-zw^2}dw\right\}
\end{equation}
\begin{align}\label{(2.31) R-S}
K_2(z,x)<\frac{e^{-z}}{2z}{\Big\{}{\Big[}\frac{35\sqrt{2}}{64}y^3+2y^2&+(\frac{105}{128z}+\frac{15}{8})\sqrt{2}y+2+\frac2z{\Big]}e^{-zy^2}\nonumber\\
 &+(\frac{105}{128z}+\frac{15}{8})\sqrt{2}\int_y^\infty e^{-zw^2}dw{\Big\}}
\end{align}
where $y=(\sqrt{x}-1/\sqrt{x})/\sqrt{2}$. If we let $x$ go to 0, then
\begin{equation}\label{K1(z)}
K_1(z)\leq\sqrt{\frac{\pi}{2z}}e^{-z}\left(1+\frac{3}{8z}\right)
\end{equation}
\begin{equation}\label{K2(z)}
K_2(z)\leq\sqrt{\frac{\pi}{2z}}e^{-z}\left(1+\frac{15}{8z}+\frac{105}{128z^2}\right)
\end{equation}

\end{lem}
\subsection{Bounds for $\psi(x)-x$ for Large Values of $x$ }
\begin{lem}[\cite{Rosser-1975}]\label{lemma 7 R-S}
Let $1<U\leq V$, and let $\Phi(y)$ be non-negative and differentiable for $U<y<V$. Let $(W-y)\Phi'(y)\geq0$ for $U<y<V$, where $W$ need not lie in $[U,V]$. Let $Y$ be one of $U, V, W$ which is neither greater than both the others nor less than both the others. Choose $j=0$ or 1 so that $(-1)^j(V-W)\geq0$. Then
\begin{align*}
\sum_{U<\gamma\leq V}\Phi(\gamma)\leq&\frac{1}{2\pi}\int_U^V\Phi(y)\log\frac{y}{2\pi}dy\\
                                     &+(-1)^j\left\{0.137+\frac{0.443}{\log Y}\right\}\int_U^V\frac{\Phi(y)}{y}dy+E_j(U,V)
\end{align*}
where the error term $E_j(U,V)$ is given by
\begin{align*}
E_j(U,V)=&\{1+(-1)^j\}R(Y)\Phi(Y)\\
&+\{N(V)-F(V)-(-1)^j R(V)\}\Phi(V)-\{N(U)-F(U)+R(U)\}\Phi(U)
\end{align*}
\end{lem}
\begin{cor}[\cite{Rosser-1975}]\label{lemma 7, corollary R-S}
If, in addition, $U>2\pi$, then
$$
\sum_{U<\gamma\leq V}\Phi(\gamma)\leq\{\frac{1}{2\pi}+(-1)^j q(Y)\}\int_U^V\Phi(y)\log\frac{y}{2\pi}dy+E_j(U,V)
$$
where
$$
q(y)=\frac{0.137\log y+0.443}{y\log y\log(y/2\pi)}
$$
\end{cor}

Define for $x\geq1$
$$
X=\sqrt{\frac{\log x}{R_0}}
$$
where $R_0=5.69693$. Also for positive $\nu$, positive integer $m$, and non-negative real $T_1$ and $T_2$, define\\

\begin{align}
R_m(\nu)=&\{(1+\nu)^{m+1}+1\}^m\label{(3.6) R-S}\\
S_1(m,\nu)=&2\sum_{\substack{\beta\leq1/2\\0<\gamma\leq T_1}}\frac{2+m\nu}{2|\rho|}\label{(3.7) R-S}\\
S_2(m,\nu)=&2\sum_{\substack{\beta\leq1/2\\\gamma>T_1}}\frac{R_m(\nu)}{\nu^m|\rho(\rho+1)\cdots(\rho+m)|}\label{(3.8) R-S}\\
S_3(m,\nu)=&2\sum_{\substack{\beta>1/2\\0<\gamma\leq T_2}}\frac{(2+m\nu)\exp(-X^2/\log\gamma)}{2|\rho|}\label{(3.9) R-S}\\
S_4(m,\nu)=&2\sum_{\substack{\beta>1/2\\\gamma>T_2}}\frac{R_m(\nu)\exp(-X^2/\log\gamma)}{\nu^m|\rho(\rho+1)\cdots(\rho+m)|}\label{(3.10) R-S}\\\nonumber
\end{align}
\begin{lem}[\cite{Rosser-1975}]
Let $T_1$ and $T_2$ be non-negative real numbers. Let $m$ be a positive integer. Let $x>1$ and $0<\delta<(x-1)/(xm)$. Then
\begin{align}\label{(3.11) R-S}
\frac1x{\Big |}\psi(x)&-\{x-\log2\pi-\frac12\log\left(1-\frac{1}{x^2}\right)\}{\Big |}\\
               &\leq\frac{1}{\sqrt{x}}\{S_1(m,\delta)+S_2(m,\delta)\}+S_3(m,\delta)+S_4(m,\delta)+\frac{m\delta}{2}
\end{align}
\end{lem}
As
$$
\frac{1}{|\rho(\rho+1)\cdots(\rho+m)|}\leq\frac{1}{\gamma^{m+1}}
$$
we can use Lemma \ref{lemma 7 R-S} to write bounds for $S_j(m,\delta)$ in terms of integrals for suitable $\Phi(y)$. We note that for $m\neq0$\\

\begin{equation}\label{(3.15) R-S}
\int_U^V y^{-(m+1)}\log\frac{y}{2\pi}dy=\frac{1+m\log(U/2\pi)}{m^2U^m}-\frac{1+m\log(V/2\pi)}{m^2V^m}
\end{equation}

In the below integral let $y=\exp(zt/2m)$
\begin{align*}
\int_U^V y^{-(m+1)}e^{-\frac{X^2}{\log y}}\log\frac{y}{2\pi}dy=&\int_{U'}^{V'}e^{-\frac{zt}{2m}(m+1)}e^{-\frac{2mX^2}{zt}}\{\frac{zt}{2m}-\log2\pi\}\frac{z}{2m}e^{\frac{zt}{2m}}dt\\
=&\frac{z}{2m}\int_{U'}^{V'}e^{-\frac{zt}{2}}e^{-\frac{2mX^2}{zt}}\{\frac{zt}{2m}-\log2\pi\}dt\\
=&\frac{z^2}{4m^2}\int_{U'}^{V'}e^{-\frac{zt}{2}-\frac{2mX^2}{zt}}tdt\\
 &-\frac{z}{2m}\log2\pi\int_{U'}^{V'}e^{-\frac{zt}{2}-\frac{2mX^2}{zt}}dt\\
=&\frac{z^2}{4m^2}\int_{U'}^{V'}e^{-\frac{z}{2}(t+\frac{4mX^2}{z^2t})}tdt\\
 &-\frac{z}{2m}\log2\pi\int_{U'}^{V'}e^{-\frac{z}{2}(t+\frac{4mX^2}{z^2t})}dt,\\
\end{align*}
where $z=2X\sqrt{m}$, $U'=(2m/z)\log U$, $V'=(2m/z)\log V$. So
\begin{align*}
\frac{z^2}{4m^2}&\int_{U'}^{V'}e^{-\frac{z}{2}(t+\frac{4mX^2}{z^2t})}tdt-\frac{z}{2m}\log2\pi\int_{U'}^{V'}e^{-\frac{z}{2}(t+\frac{4mX^2}{z^2t})}dt\\
&=\frac{z^2}{2m^2}\{K_2(z,U')-K_2(z,V')\}-\frac{z}{m}\log2\pi\{K_1(z,U')-K_1(z,V')\}
\end{align*}
Hence,
\begin{align}\label{(3.16) R-S}
\int_U^V &y^{-(m+1)}e^{-\frac{X^2}{\log y}}\log\frac{y}{2\pi}dy\\
&=\frac{z^2}{2m^2}\{K_2(z,U')-K_2(z,V')\}-\frac{z}{m}\log2\pi\{K_1(z,U')-K_1(z,V')\}\nonumber
\end{align}
Also if we let $y=\exp(X^2/t)$, we get
\begin{align}
\int_U^V y^{-1}e^{-\frac{X^2}{\log y}}\log\frac{y}{2\pi}dy=&\int_{U''}^{V''} e^{-\frac{X^2}{t}}e^{-t}\{\frac{X^2}{t}-\log2\pi\}\left(-\frac{X^2}{t^2}e^\frac{X^2}{t}\right)dt\nonumber\\
=&-X^4\int_{U''}^{V''} t^{-3}e^{-t}dt+X^2\log2\pi\int_{U''}^{V''} t^{-2}e^{-t}dt\nonumber\\
=&X^4\{\Gamma(-2,V'')-\Gamma(-2,U'')\}\nonumber\\
&-X^2\log2\pi\{\Gamma(-1,V'')-\Gamma(-1,U'')\},\label{(3.17) R-S}\\\nonumber
\end{align}
where $U''=X^2/\log U$, $V''=X^2/\log V$ and
$$
\Gamma(a,x)=\int_x^\infty t^{a-1}e^{-t}dt
$$
is incomplete gamma function.
\begin{thm}\label{bound for psi x-x}
If $\log x>110$, then
$$
|\psi(x)-x|<x\varepsilon(x),
$$
$$
|\vartheta(x)-x|<x\varepsilon(x),
$$
where
\begin{equation}\label{(3.19) R-S}
\varepsilon(x)=1.062253\left(1-\frac{0.900377}{2X}\right)X^{3/4}e^{-X},
\end{equation}
and
$$
X=\sqrt{\frac{\log x}{R_0}},\qquad R_0=5.69693.
$$
\end{thm}
\begin{proof}
Take $m=1$ and $T_1=T_2=0$ in (\ref{(3.6) R-S}) through (\ref{(3.11) R-S}). By Lemma \ref{lemma 8 of Rosser},
$$
S_1(1,\delta)+S_2(1,\delta)<(0.0463)\frac{2+2\delta+\delta^2}{\delta}.
$$
Also, as $\beta=\frac12$ for $|\gamma|\leq A$, and the zeros off the critical line occur in pairs which are symmetrical with respect to this line, we have
$$
S_3(1,\delta)+S_4(1,\delta)\leq\frac{2+2\delta+\delta^2}{\delta}\sum_{\gamma>A}\phi_1(\gamma),
$$
where
$$
\phi_m(y)=\frac{e^{-X^2/\log y}}{y^{m+1}}.
$$
We appeal to Corollary \ref{lemma 7, corollary R-S} with $\Phi(y)=\phi_1(y)$, $j=0$, $U=A$, $V=\infty$, and $W=W_1$, where for $m>-1$
$$
W_m=\exp(X/\sqrt{m+1}).
$$
Not that $q(Y)\leq q(A)$. Also, as $N(A)=F(A)$, we have
\begin{equation}\label{(3.25) R-S}
E_0=2R(Y)\phi_1(Y)-R(A)\phi_1(A).
\end{equation}
Since $K_\nu(z,x)+K_{-\nu}(z,1/x)=K_\nu(z)$ and (\ref{(3.16) R-S}), we have
$$
\int_A^\infty\phi_1(y)\log\frac{y}{2\pi}dy\leq2X\{XK_2(2X)-\log(2\pi)K_1(2X)\}\leq2X^2K_2(2X).
$$
Then, by (\ref{K1(z)}) and (\ref{K2(z)}), we conclude
\begin{align}\label{(3.26) R-S}
\sum_{\gamma>A}\phi_1(\gamma)\leq&\{\frac{1}{2\pi}+\frac{0.137\log A+0.443}{A\log A\log(A/2\pi)}\}\int_A^\infty\phi_1(y)\log\frac{y}{2\pi}dy+E_0\nonumber\\
<&\{\frac{1}{2\pi}+\frac{0.137\log A+0.443}{A\log A\log(A/2\pi)}\}(2X^2)\{K_2(2X)\}+E_0\nonumber\\
\leq&\{\frac{1}{2\pi}+\frac{0.137\log A+0.443}{A\log A\log(A/2\pi)}\}(2X^2)\{1+\frac{15}{16X}+\frac{105}{512X^2}\}\sqrt{\frac{\pi}{4X}}e^{-2X}+E_0\nonumber\\
<&(0.28209479177389)\{1+\frac{15}{16X}+\frac{105}{512X^2}\}X^{3/2}e^{-2X}+E_0.
\end{align}
If $W_1\leq A$, then $Y=A$. Then by (\ref{(3.25) R-S})
$$
E_0=R(A)\phi_1(A)=\frac{R(A)}{A^2}\{e^{-X^2/\log A}X^{1/2}e^{2X}\}X^{-1/2}e^{-2X}.
$$
As the expression
$$
\exp\{-\frac{X^2}{\log A}+\frac12\log X+2X\}
$$
takes its maximum at
$$
X=\frac12\log A+\frac12\sqrt{\log^2A+\log A}
$$
we conclude that

\begin{equation}\label{(3.27) R-S}
E_0<1.53\cdot10^{-11}X^{-1/2}e^{-2X}.
\end{equation}

If $W_1>A$, then $Y=W_1$ and $X>40$. As $R(y)/\log y$ is decreasing for $y>e^e$, (\ref{(3.25) R-S}) gives
\begin{align*}
E_0<2R(Y)\phi_1(Y)=&2\frac{R(Y)}{\log Y}\phi_1(Y)\log Y\\
                  <&2\frac{R(A)}{\log A}\phi_1(W_1)\log W_1\\
                  =&2\frac{R(A)}{\log A}\frac{X}{\sqrt{2}}e^{-2\sqrt{2}X}\\
                  =&\sqrt{2}\frac{R(A)}{\log A}Xe^{-2\sqrt{2}X}\\
                  =&\sqrt{2}\frac{R(A)}{\log A}X^{3/2}e^{2X-2\sqrt{2}X}X^{-1/2}e^{-2X}\\
                  <&(3.56\cdot10^{-13})X^{-1/2}e^{-2X}.
\end{align*}
so that we conclude (\ref{(3.27) R-S}) for this case also. Then by (\ref{(3.26) R-S})
\begin{equation}\label{(3.28) R-S}
\sum_{\gamma>A}\phi_1(\gamma)<(0.282094791774)\{1+\frac{15}{16X}+\frac{105}{512X^2}\}X^{3/2}e^{-2X}.
\end{equation}
As $\log x\geq110$
\begin{equation}\label{(3.29) R-S}
\frac{0.0463}{\sqrt{x}}=(0.0463)e^{-\frac12R_0X^2}<10^{-21}X^{-1/2}e^{-2X}.
\end{equation}
Choose
\begin{equation}\label{(3.30) R-S}
\delta=2(0.282094791775)^{1/2}\{1+\frac{15}{32X}\}X^{3/4}e^{-X}.
\end{equation}
So
\begin{align*}
&\frac{S_1(1,\delta)+S_2(1,\delta)}{\sqrt{x}}+S_3(1,\delta)+S_4(1,\delta)<\frac{2+2\delta+\delta^2}{\delta}\left(\frac{0.0463}{\sqrt{x}}+\sum_{\gamma>A}\phi_1(\gamma)\right)\\
&<\frac{2+2\delta+\delta^2}{\delta}\left(10^{-21}X^{-1/2}e^{-2X}+(0.282094791774)\{1+\frac{15}{16X}+\frac{105}{512X^2} \}X^{3/2}e^{-2X}\right)\\
&<\frac{2+2\delta+\delta^2}{\delta}(0.282094791775)\{1+\frac{15}{16X}+\frac{105}{512X^2} \}X^{3/2}e^{-2X}\\
&<\frac2\delta(0.282094791775)\left(1+\frac{15}{32X}\right)^2X^{3/2}e^{-2X}.\\
\end{align*}
Let
$$
\varepsilon_1(x)=\frac{S_1(1,\delta)+S_2(1,\delta)}{\sqrt{x}}+S_3(1,\delta)+S_4(1,\delta)+\frac\delta2.
$$
Then
\begin{align*}
\varepsilon_1(x)<&\frac2\delta(0.282094791775)\left(1+\frac{15}{32X}\right)^2X^{3/2}e^{-2X}+\frac\delta2\\
=&(0.282094791775)^{1/2}\{1+\frac{15}{32X}\}X^{3/4}e^{-X}+(0.282094791775)^{1/2}\{1+\frac{15}{32X}\}X^{3/4}e^{-X}\\
=&2(0.282094791775)^{1/2}\{1+\frac{15}{32X}\}X^{3/4}e^{-X}\\
<&(1.06225193203)\{1+\frac{15}{32X}\}X^{3/4}e^{-X}.
\end{align*}
So
$$
|\psi(x)-\{x-\log2\pi-\frac12\log(1-\frac{1}{x^2})\}|<x\varepsilon(x),
$$
where
$$
\varepsilon(x)=(1.06225193203)\{1+\frac{15}{32X}\}X^{3/4}e^{-X}.
$$
By Theorem 13 of \cite{Rosser-1962},
$$
|\psi(x)-\vartheta(x)|<1.43\sqrt{x}
$$
Thus, it would appear that for $\vartheta(x)$ we should increase $\varepsilon(x)$ by $1.43/\sqrt{x}$. However,
we can treat it as in (\ref{(3.29) R-S}) to show it is absorbed when we round up some of the coefficients.
\end{proof}
\begin{thm}\label{theorem 3 R-S}
If $\log x\geq110$
$$
|\psi(x)-x|<x\varepsilon^\ast(x),
$$
where
$$
\varepsilon^\ast(x)=\frac{\varepsilon(x)}{\sqrt{2}}\left\{1+\frac{3\log X}{2\sqrt{\pi(4X-3\log X)}}+\frac{3}{2\sqrt{\pi X}}\right\}
$$
and
$$
r(x)=1+\frac{15}{32X}.
$$
\end{thm}
\begin{proof}
Take
\begin{equation}\label{(3.35) R-S}
\delta=\frac{1}{\sqrt{2}}(1.06225193203)\{1+\frac{15}{32X}\}X^{3/4}e^{-X}.
\end{equation}
We may assume $X\geq33.4$ or (more accurate $x\geq e^{6344}$), since if $X\leq33.36$ or (more accurate $x\leq e^{6343}$), $\varepsilon^\ast(x)>\varepsilon(x)$. We take $m=1$, $T_1=0$, and

\begin{equation}\label{(3.36) R-S}
T_2=X^{-3/4}e^X.
\end{equation}

As $X>32$ we have $A<T_2<e^X=W_0$ and $W_1<T_2$.

We can treat $\{S_1(1,\delta)+S_2(1,\delta)\}/\sqrt{x}$ and the error terms $E_j(U,V)$ arising from the use of Corollary \ref{lemma 7, corollary R-S}, as we did in the proof of the previous theorem. Thus we can proceed as though

\begin{equation}\label{(3.37) R-S}
S_3(1,\delta)=\frac{2+\delta}{2}\sum_{0<\gamma\leq T_2}\phi_0(\gamma)<\frac{2+\delta}{2}\left(\frac{1}{2\pi}-q(T_2)\right)\int_A^{T_2}\phi_0(y)\log\frac{y}{2\pi}dy.
\end{equation}
If $\nu\leq1$ and $x>0$, then
$$
\Gamma(\nu,x)=\int_x^\infty t^{\nu-1}e^{-t}dt\leq x^{\nu-1}\int_x^\infty e^{-t}dt=x^{\nu-1}e^{-x}
$$
and
$$
\Gamma(\nu,x)=\int_x^\infty t^{\nu-1}e^{-t}dt\geq (1+x)^{\nu-1}e^{-x}.
$$
Hence, by (\ref{(3.17) R-S}), we have in effect
\begin{align}\label{(3.38) R-S}
S_3(1,\delta)<&\frac{2+\delta}{4\pi}\left\{X^4\{\Gamma(-2,V'')-\Gamma(-2,U'')\}
-X^2\log2\pi\{\Gamma(-1,V'')-\Gamma(-1,U'')\}\right\}\nonumber\\
\leq&\frac{2+\delta}{4\pi}\left\{X^4\Gamma(-2,V'')-X^2\log2\pi\Gamma(-1,V'')\right\}\nonumber\\
<&\frac{2+\delta}{4\pi}\left\{X^4(V'')^{-3}-X^2\log2\pi(1+V'')^{-2}\right\}e^{-V''},
\end{align}
where
\begin{equation}\label{(3.39) R-S}
V''=\frac{X^2}{\log T_2}=\frac{4X^2}{4X-3\log X}.
\end{equation}
Then
\begin{equation}\label{(3.40) R-S}
V''>X+\frac34\log X,\qquad e^{-V''}<X^{-3/4}e^{-X}.
\end{equation}
Also
\begin{align*}
X^4(V'')^{-3}-X^2\log2\pi(V'')^{-2}=&\left(1-\frac{3\log X}{4X}\right)^2\\
&\left(X-\frac{3}{4}\log X-\frac{\log2\pi}{(1+1/X-3\log(X)/(4X^2))^2}\right)<X.
\end{align*}
So, effectively

\begin{equation}\label{(3.41) R-S}
S_3(1,\delta)<\frac{2+\delta}{4\pi}X^{1/4}e^{-X}.
\end{equation}

Similarly, we can proceed as though

\begin{equation}\label{(3.42) R-S}
S_4(1,\delta)=\frac{2+2\delta+\delta^2}{\delta}\sum_{\gamma> T_2}\phi_1(\gamma)<\frac{2+2\delta+\delta^2}{\delta}\left(\frac{1}{2\pi}+q(T_2)\right)\int_{T_2}^\infty\phi_1(y)\log\frac{y}{2\pi}dy.
\end{equation}

By (\ref{(3.16) R-S})
\begin{equation}\label{(3.43) R-S}
\int_{T_2}^\infty\phi_1(y)\log\frac{y}{2\pi}dy=2\{X^2K_2(2X,U')-X\log(2\pi)K_1(2X,U')\}\leq2X^2K_2(2X,U'),
\end{equation}
where
\begin{equation}\label{(3.44) R-S}
U'=\frac1X \log T_2=1-\frac{3\log X}{4X}.
\end{equation}
Write temporarily
\begin{equation}\label{(3.45) R-S}
q=\frac{3\log X}{4X},\qquad y=\frac{1}{\sqrt{2}}\left(\sqrt{U'}-\frac{1}{\sqrt{U'}}\right).
\end{equation}
Then $y$ is negative, and
$$
y^2=\frac{1}{2}\left(U'+\frac{1}{U'}\right)-1=\frac{q^2}{2(1-q)}.
$$
So, by splitting the integral in Lemma \ref{upper bound of K1(z,x), K2(z,x)} at $w=0$, we get
\begin{align}\label{(3.46) R-S}
\sqrt{2}\int_y^\infty e^{-2Xw^2}dw=&\sqrt{2}\int_y^0 e^{-2Xw^2}dw+\sqrt{2}\int_0^\infty e^{-2Xw^2}dw\nonumber \\
=& \sqrt{2}\int_y^0 e^{-2Xw^2}dw+\frac{\sqrt{\pi}}{2\sqrt{X}}=\sqrt{2}\int_0^y e^{-2Xw^2}dw+\frac{\sqrt{\pi}}{2\sqrt{X}}\nonumber\\
<& \sqrt{2}\int_0^y 1 dw+\frac{\sqrt{\pi}}{2\sqrt{X}}=\frac{q}{\sqrt{1-q}}+\frac{\sqrt{\pi}}{2\sqrt{X}}.
\end{align}
Hence, by (\ref{(2.30) R-S}) we get
\begin{align}\label{(3.47) R-S}
X\log(2\pi)&K_1(2X,U')\nonumber\\
&<X\log(2\pi)\frac{e^{-2X}}{4X}\left\{\left(1+\frac{3\sqrt{2}}{8}y\right)e^{-2Xy^2}+\left(\frac38+2X\right)\left(\frac{\sqrt{\pi}}{2\sqrt{X}}+\frac{q}{\sqrt{1-q}}\right)\right\}\nonumber\\
&\leq\frac14\log(2\pi)e^{-2X}\left\{1+\left(\frac38+2X\right)\left(\frac{\sqrt{\pi}}{2\sqrt{X}}+\frac{q}{\sqrt{1-q}}\right)\right\}\nonumber\\
&=\frac{\sqrt{\pi}}{4}\log(2\pi)X^{3/2}e^{-2X}\left\{\frac{1}{\sqrt{\pi}X^{3/2}}+\left(\frac{3}{16X^2}+\frac1X\right)\left(1+\frac{2q\sqrt{X}}{\sqrt{\pi(1-q)}}\right)\right\}
\end{align}
As $1+zy^2<e^{zy^2}$, we have $(2y^2+2/z)e^{-zy^2}<2/z=1/X$. Hence, by (\ref{(2.31) R-S}) we get
\begin{align}\label{(3.48) R-S}
X^2&K_2(2X,U')\nonumber\\
<&X^2\frac{e^{-2X}}{4X}{\Big\{}\left[\frac{35\sqrt{2}}{64}y^3+2y^2+\left(\frac{105}{128z}+\frac{15}{8}\right)\sqrt{2}y+2+\frac2z\right]e^{-zy^2}\nonumber\\
 &+\left(\frac{105}{256X}+\frac{15}{8}+2X\right)\left(\frac{\sqrt{\pi}}{2\sqrt{X}}+\frac{q}{\sqrt{1-q}}\right){\Big\}}\nonumber\\
<&\frac14Xe^{-2X}\left\{2+\frac1X+\left(\frac{105}{256X}+\frac{15}{8}+2X\right)\left(\frac{\sqrt{\pi}}{2\sqrt{X}}+\frac{q}{\sqrt{1-q}}\right)\right\}\nonumber\\
=&\frac14\sqrt{\pi}X^{3/2}e^{-2X}\left\{\frac{2}{\sqrt{\pi X}}+\frac{1}{\sqrt{\pi}X^{3/2}}+\left(\frac{105}{512X^2}+\frac{15}{16X}+1\right)\left(1+\frac{2q\sqrt{X}}{\sqrt{\pi(1-q)}}\right)\right\}.
\end{align}
Combining with (\ref{(3.43) R-S}) and (\ref{(3.47) R-S}) gives
$$
\int_{T_2}^\infty \phi_1(y)\log\frac{y}{2\pi}dy<\frac{\sqrt{\pi}}{2}X^{3/2}e^{-2X}Q_1,
$$
where
\begin{align*}
Q_1=\frac{2}{\sqrt{\pi X}}+\frac{1}{\sqrt{\pi}X^{3/2}}+\left(\frac{105}{512X^2}+\frac{15}{16X}+1\right)\left(1+\frac{2q\sqrt{X}}{\sqrt{\pi(1-q)}}\right).
\end{align*}
So finally by (\ref{(3.42) R-S}) and (\ref{(3.35) R-S})
\begin{align}\label{(3.49) R-S}
S_4(1,\delta)<&\frac{2+2\delta+\delta^2}{\delta}\left(\frac{1}{2\pi}+q(T_2)\right)\int_{T_2}^\infty\phi_1(y)\log\frac{y}{2\pi}dy\nonumber\\
             <&\frac{2+2\delta+\delta^2}{\delta}\left(\frac{1}{2\pi}+q(T_2)\right)\frac{\sqrt{\pi}}{2}X^{3/2}e^{-2X}Q_1\nonumber\\
             =&\frac{2+2\delta+\delta^2}{\delta}\left(\frac{1}{2\pi}+q(T_2)\right)\frac{\sqrt{\pi}}{2}X^{3/2}e^{-2X}2r(x)^2\nonumber\\
             &\times{\Big\{}\frac{1}{r(x)^2}\left(\frac{1}{\sqrt{\pi X}}+\frac{1}{2\sqrt{\pi}X^{3/2}}\right)\nonumber\\
             &\quad+\frac{1}{2r(x)^2}\left(\frac{105}{512X^2}+\frac{15}{16X}+1\right)\left(1+\frac{2q\sqrt{X}}{\sqrt{\pi(1-q)}}\right){\Big\}}\nonumber\\
             <&2\frac{r(x)^2}{\delta}\left(\frac{1}{2\pi}+q(T_2)\right)\frac{\sqrt{\pi}}{2}X^{3/2}e^{-2X}\nonumber\\
             &\times\left\{\frac{2+2\delta+\delta^2}{r(x)^2}\left(\frac{1}{\sqrt{\pi X}}+\frac{1}{2\sqrt{\pi}X^{3/2}}\right)+\left(1+\frac{2q\sqrt{X}}{\sqrt{\pi(1-q)}}\right)\right\}\nonumber\\
             =&2\frac{r(x)^2}{\delta}\left(\frac{1}{2\pi}+q(T_2)\right)\frac{\sqrt{\pi}}{2}X^{3/2}e^{-2X}Q_2\nonumber\\
             =&2\frac{r(x)^2}{\delta^2}\delta\left(\frac{1}{2\pi}+q(T_2)\right)\frac{\sqrt{\pi}}{2}X^{3/2}e^{-2X}Q_2\nonumber\\
             <&(\frac12)\delta Q_2,\qquad(X\geq28),
\end{align}
where

\begin{equation}\label{(3.50) R-S}
Q_2=1+\frac{2q\sqrt{X}}{\sqrt{\pi(1-q)}}+\frac{2+2\delta+\delta^2}{r(x)^2}\left(\frac{1}{\sqrt{\pi X}}+\frac{1}{2\sqrt{\pi}X^{3/2}}\right).
\end{equation}
So
\begin{align*}
\frac{S_1(1,\delta)+S_2(1,\delta)}{\sqrt{x}}&+S_3(1,\delta)+S_4(1,\delta)+\frac\delta2\\
<&\frac{2+2\delta+\delta^2}{\delta}\frac{0.0463}{\sqrt{x}}+\frac{2+\delta}{4\pi}X^{1/4}e^{-X}\\
&+\frac\delta2\left\{1+\frac{2q\sqrt{X}}{\sqrt{\pi(1-q)}}+\frac{2+2\delta+\delta^2}{r(x)^2}\left(\frac{1}{\sqrt{\pi X}}+\frac{1}{2\sqrt{\pi}X^{3/2}}\right)\right\}+\frac\delta2\\
<&\frac\delta2\left\{1+\frac{2q\sqrt{X}}{\sqrt{\pi(1-q)}}+\frac{2+2\delta+\delta^2}{r(x)^2}\left(\frac{1}{\sqrt{\pi X}}+\frac{1}{2\sqrt{\pi X}}\right)\right\}+\frac\delta2,\qquad(X\geq6)\\
=&\delta\left\{1+\frac{q\sqrt{X}}{\sqrt{\pi(1-q)}}+\frac{2+2\delta+\delta^2}{2r(x)^2}\left(\frac{1}{\sqrt{\pi}}+\frac{1}{2\sqrt{\pi}}\right)\frac{1}{\sqrt{X}}\right\}\\
<&\delta\left\{1+\frac{q\sqrt{X}}{\sqrt{\pi(1-q)}}+\frac{3}{2\sqrt{\pi X}}\right\}\\
=&\frac{\varepsilon(x)}{\sqrt{2}}\left\{1+\frac{q\sqrt{X}}{\sqrt{\pi(1-q)}}+\frac{3}{2\sqrt{\pi X}}\right\}\\
=&\frac{\varepsilon(x)}{\sqrt{2}}\left\{1+\frac{3\log X}{2\sqrt{\pi(4X-3\log X)}}+\frac{3}{2\sqrt{\pi X}}\right\}.
\end{align*}
\end{proof}
\subsection{Numerical Bounds for $\psi(x)-x$ for Moderate Values of $x$}
In our main table (at the end of the thesis) we tabulate values of $\varepsilon$ against $b$. These have been determined so that if $x\geq e^b$, then

\begin{equation}\label{(4.1) R-S}
|\psi(x)-x|<\varepsilon x.
\end{equation}

Let $T_2=0$ and

\begin{equation}\label{(4.2) R-S}
T_1=\frac1\delta\left(\frac{2R_m(\delta)}{2+m\delta}\right)^{1/m}.
\end{equation}
We chose also

\begin{equation}\label{(4.3) R-S}
D=963.5670402.
\end{equation}

The zeros for which $0<\gamma\leq D$ are exactly 620 in number.

\begin{equation}\label{(4.4) R-S}
S\equiv\sum_{0<\gamma\leq D}\frac{1}{|\rho|}=\sum_{0<\gamma\leq D}\frac{1}{(\gamma^2+1/4)^{1/2}}<2.
\end{equation}

\begin{lem}
With $T_1$ and $D$ given by (\ref{(4.2) R-S}) and (\ref{(4.3) R-S}), if $T_1\geq D$, $\delta>0$, and $m$ is a positive integer, then

\begin{equation}\label{(4.5) R-S}
S_1(m,\delta)+S_2(m,\delta)<\frac{2+m\delta}{4\pi}\left\{\left(\log\frac{T_1}{2\pi}+\frac1m\right)^2+\frac{1}{m^2}-0.1580304-2.531837599\frac{m}{(m+1)T_1}\right\}.
\end{equation}
\end{lem}

\begin{proof}
By (\ref{(3.7) R-S}) and (\ref{(4.4) R-S})
$$
S_1(m,\delta)=(2+m\delta)\sum_{\substack{\beta\leq1/2\\0<\gamma\leq T_1}}\frac1{|\rho|}<(2+m\delta)\left\{S+\sum_{\substack{D<\gamma\leq T_1}}\frac1{\gamma}\right\}.
$$
Taking $\Phi(y)=y^{-1}$, $j=0$, $U=D$, $V=T_1$, and $W=0$ in Lemma \ref{lemma 7 R-S} gives
\begin{align*}
\sum_{D<\gamma\leq T_1}\frac1{\gamma}\leq&\frac{1}{2\pi}\int_D^{T_1}\frac{1}{y}\log\frac{y}{2\pi}dy+\left\{0.137+\frac{0.443}{\log D}\right\}\int_D^{T_1}\frac{1}{y^2}dy+E_0\\
\leq&\frac{1}{4\pi}\left\{\log^2\frac{T_1}{2\pi}-\log^2\frac{D}{2\pi}\right\}+(0.137+\frac{0.443}{\log D})\left\{-\frac{1}{T_1}+\frac{1}{D}\right\}+E_0.
\end{align*}
Then (\ref{(4.4) R-S}) together with $N(D)=620$ gives
\begin{align}\label{(4.6) R-S}
\frac{S_1(m,\delta)}{2+m\delta}<&S+\frac{1}{4\pi}\left\{\log^2\frac{T_1}{2\pi}-\log^2\frac{D}{2\pi}\right\}-(0.137+\frac{0.443}{\log D})\left\{\frac{1}{T_1}-\frac{1}{D}\right\}\nonumber\\
&+\frac{1}{T_1}(N(T_1)-F(T_1)-R(T_1))-\frac1D(N(D)-F(D)-R(D))\nonumber\\
<&2-\frac{1}{4\pi}\log^2\frac{D}{2\pi}+\frac1D\left(0.137+\frac{0.443}{\log D}-N(D)+F(D)+R(D)\right)\nonumber\\
&+\frac{1}{4\pi}\log^2\frac{T_1}{2\pi}-\frac1{T_1}\left(0.137+\frac{0.443}{\log D}-N(T_1)+F(T_1)+R(T_1)\right)\nonumber\\
<&-0.01257566+\frac{1}{4\pi}\log^2\frac{T_1}{2\pi}-\frac1{T_1}\left(0.137+\frac{0.443}{\log D}-N(T_1)+F(T_1)+R(T_1)\right).
\end{align}
Since
$$
S_2(m,\delta)=2\sum_{\substack{\beta\leq1/2\\\gamma>T_1}}\frac{R_m(\delta)}{\delta^m|\rho(\rho+1)\cdots(\rho+m)|}.
$$
Taking $\Phi(y)=y^{-(m+1)}$, $j=0$, $U=T_1$, $V=\infty$, and $W=0$ in Lemma \ref{lemma 7 R-S} and using (\ref{(3.15) R-S}) gives
\begin{align*}
\frac{\delta^mS_2(m,\delta)}{2R_m(\delta)}=&\sum_{\substack{\beta\leq1/2\\\gamma>T_1}}\frac{1}{|\rho(\rho+1)\cdots(\rho+m)|}<\sum_{\gamma>T_1}\frac{1}{\gamma^{m+1}}\\
\leq&\frac{1}{2\pi}\int_{T_1}^\infty y^{-(m+1)}\log\frac{y}{2\pi}dy+\left(0.137+\frac{0.443}{\log T_1}\right)\int_{T_1}^\infty\frac{y^{-(m+1)}}{y}dy+E_0^\ast\\
=&\frac{1}{2\pi}\left(\frac{1}{m^2T_1^m}+\frac{1}{mT_1^m}\log\frac{T_1}{2\pi}\right)+\left(0.137+\frac{0.443}{\log T_1}\right)\frac{1}{(m+1)T_1^{m+1}}+E_0^\ast\\
=&\frac{1}{T_1^m}\left\{\frac{1}{2\pi m}\log\frac{T_1}{2\pi}+\frac{1}{2\pi m^2}+\left(0.137+\frac{0.443}{\log T_1}\right)\frac{1}{(m+1)T_1}\right\}+E_0^\ast\\
=&\frac{1}{T_1^m}\left\{\frac{1}{2\pi m}\log\frac{T_1}{2\pi}+\frac{1}{2\pi m^2}+\left(0.137+\frac{0.443}{\log T_1}\right)\frac{1}{(m+1)T_1}++E_0^\ast T_1^m\right\}.\\
\end{align*}
Using (\ref{(4.2) R-S}) and combining with (\ref{(4.6) R-S}) gives
\begin{align*}
S_1&(m,\delta)+S_2(m,\delta)\\
<&(2+m\delta)\left[-0.01257566+\frac{1}{4\pi}\log^2\frac{T_1}{2\pi}-\frac1{T_1}\left(0.137+\frac{0.443}{\log D}-N(T_1)+F(T_1)+R(T_1)\right)\right]\\
&+\frac{2R_m(\delta)}{\delta^m}\frac{1}{T_1^m}\left[\frac{1}{2\pi m}\log\frac{T_1}{2\pi}+\frac{1}{2\pi m^2}+\left(0.137+\frac{0.443}{\log T_1}\right)\frac{1}{(m+1)T_1}+E_0^\ast T_1^m\right]\\
=&(2+m\delta)\left[-0.01257566+\frac{1}{4\pi}\log^2\frac{T_1}{2\pi}-\frac1{T_1}\left(0.137+\frac{0.443}{\log D}-N(T_1)+F(T_1)+R(T_1)\right)\right]\\
&+(2+m\delta)\left[\frac{1}{2\pi m}\log\frac{T_1}{2\pi}+\frac{1}{2\pi m^2}+\left(0.137+\frac{0.443}{\log T_1}\right)\frac{1}{(m+1)T_1}+E_0^\ast T_1^m\right]\\
=&\frac{2+m\delta}{4\pi}\left[\log^2\frac{T_1}{2\pi}+4\pi\left\{-0.01257566-\frac1{T_1}\left(0.137+\frac{0.443}{\log D}-N(T_1)+F(T_1)+R(T_1)\right)\right\}\right]\\
&+\frac{2+m\delta}{4\pi}\left[\frac{2}{m}\log\frac{T_1}{2\pi}+\frac{2}{m^2}+4\pi\left(0.137+\frac{0.443}{\log T_1}\right)\frac{1}{(m+1)T_1}+4\pi E_0^\ast T_1^m\right]\\
=&\frac{2+m\delta}{4\pi}\left\{\left(\log\frac{T_1}{2\pi}+\frac1m\right)^2+\frac{1}{m^2}+J\right\},
\end{align*}
where
\begin{align*}
J=&4\pi\left\{-0.01257566-\frac1{T_1}\left(0.137+\frac{0.443}{\log D}-N(T_1)+F(T_1)+R(T_1)\right)\right\}\\
&+4\pi\left(0.137+\frac{0.443}{\log T_1}\right)\frac{1}{(m+1)T_1}+4\pi E_0^\ast T_1^m\\
=&4\pi\left\{-0.01257566-\frac1{T_1}\left(0.137+\frac{0.443}{\log D}\right)+\left(0.137+\frac{0.443}{\log T_1}\right)\frac{1}{(m+1)T_1}\right\}\\
<&4\pi\left\{-0.01257566-\frac1{T_1}\left(0.137+\frac{0.443}{\log D}\right)\left(1-\frac{1}{m+1}\right)\right\}\\
<&-0.1580304-2.531837599\frac{m}{(m+1)T_1}.
\end{align*}
\end{proof}

\begin{thm}
Let $T_1\geq D$. Let $m$ be a positive integer, let $\Omega_1$ denote the right side of (\ref{(4.5) R-S}) and let
\begin{align}\label{(4.7) R-S}
\Omega_2=&(0.159155)\frac{R_m(\delta)z}{2m^2\delta^m}\left\{zK_2(z,A')-2m\log(2\pi)K_1(z,A')\right\}\nonumber\\
&+\frac{R_m(\delta)}{\delta^m}\{2R(Y)\phi_m(Y)-R(A)\phi_m(A)\},
\end{align}
where $z=2X\sqrt{m}=2\sqrt{mb/R_0}$, $A'=(2m/z)\log A$, $Y=\max\{A,\exp\sqrt{b/(m+1)R_0}\}$. If $b>1/2$ and $0<\delta<(1-e^{-b})/m$, then
$$
|\psi(x)-x|<\varepsilon x,\qquad(x\geq e^b),
$$
where
$$
\varepsilon=\Omega_1 e^{-b/2}+\Omega_2+\frac{m\delta}{2}+e^{-b}\log2\pi.
$$
\end{thm}
\begin{proof}
Take $T_2=0$, then $S_3(m,\delta)=0$ and by Corollary \ref{lemma 7, corollary R-S} and \ref{(3.16) R-S},
\begin{align*}
\frac{\delta^m}{R_m(\delta)}S_4(m,\delta)=&\sum_{\substack{\beta>1/2\\ \gamma>0}}\frac{e^{-X^2/\log\gamma}}{|\rho(\rho+1)\cdots(\rho+m)|}<\sum_{\gamma>A}\frac{e^{-X^2/\log\gamma}}{\gamma^{m+1}}=\sum_{\gamma>A}\phi_m(\gamma)\\
\leq&\left\{1+q(Y)\right\}\int_A^\infty\phi_m(y)\log\frac{y}{2\pi}dy+E_0\\
<&0.159155\left(\frac{z^2}{2m^2}K_2(z,A')-\frac{z}{m}\log(2\pi)K_1(z,A')\right)\\
&+\{2R(Y)\phi_m(Y)-R(A)\phi_m(A)\}.\\
\end{align*}
So
$$
S_3(m,\delta)+S_4(m,\delta)<\Omega_2.
$$
Since
$$
\frac1x|\psi(x)-x+\log2\pi+\frac12\log(1-\frac{1}{x^2})|<\frac{S_1(m,\delta)+S_2(m,\delta)}{\sqrt{x}}+S_3(m,\delta)+S_4(m,\delta)+\frac{m\delta}2.
$$
So
\begin{align*}
\frac1x|\psi(x)-x|<&\frac{S_1(m,\delta)+S_2(m,\delta)}{\sqrt{x}}+S_3(m,\delta)+S_4(m,\delta)+\frac{m\delta}2+\frac{\log2\pi}x\\
                  <&\frac{\Omega_1}{\sqrt{x}}+\Omega_2+\frac{m\delta}2+\frac{\log2\pi}x\\
                  \leq&\Omega_1 e^{-b/2}+\Omega_2+\frac{m\delta}{2}+e^{-b}\log2\pi.
\end{align*}
\end{proof}

\begin{thm}
Let $T_1\geq D$ and $A\leq T_2\leq \exp\sqrt{b/R}$. Let $m$ be a positive integer and let
\begin{align}\label{(4.9) R-S}
\Omega_3=&\frac{2+m\delta}{4\pi}\left[X^4\left\{\Gamma(-2,T'')-\Gamma(-2,A'')\right\}-X^2\log(2\pi)\left\{\Gamma(-1,T'')-\Gamma(-1,A'')\right\}\right]\nonumber\\
&+\frac{2+m\delta}2\left\{2R(T_2)\phi_0(T_2)-R(A)\phi_0(A)\right\}+\Omega_2^\ast,
\end{align}
where $A''=b/(R_0\log A)$, $T''=b/(R_0\log T_2)$, and $\Omega_2^\ast$ is obtained from $\Omega_2$ by deleting the term $-R(A)\phi_m(A)$ in (\ref{(4.7) R-S}) and then replacing $A$ by $T_2$ in the definition of $A'$ and $Y$. If $b>1/2$ and $0<\delta<(1-e^{-b})/m$, then (\ref{(4.1) R-S}) holds for all $x\geq e^b$, where

\begin{equation}\label{(4.10) R-S}
\varepsilon=\Omega_1e^{-b/2}+\Omega_3+\frac{m\delta}2+e^{-b}\log2\pi.
\end{equation}
\end{thm}
If we use the following bounds for $\Gamma(\nu,x)$ and $\nu<1$, for large $b$ we get a better bounds  than those given in three theorems before the last one;
$$
\frac{x^\nu e^{-x}}{x+1-\nu}<\Gamma(\nu,x)<x^{\nu-3}e^{-x}\{x^2+(\nu-1)x+(\nu-1)(\nu-2)\},\qquad(x>0,\ \nu<1).
$$
\subsection{Bounds for $\vartheta(x)-x$ for Large Values of $x$}
\begin{thm}[\cite{Rosser-1975}]
We have
\begin{align*}
\vartheta(x)<&1.001,102x,\hspace{67pt}(x>0),\\
0.998,684x<&\vartheta(x),\hspace{100pt}(x\geq1,319,007),\\
\psi(x)-\vartheta(x)<&1.001,102\sqrt{x}+3\sqrt[3]{x},\hspace{19pt}(x>0),\\
0.998,684\sqrt{x}<&\psi(x)-\vartheta(x),\hspace{62pt}(x\geq121).\\
\end{align*}
\end{thm}
\begin{cor}[\cite{Rosser-1975}]
We have
\begin{align*}
\vartheta(x)>&0.998x,\qquad(x\geq487,381),\\
\vartheta(x)>&0.995x,\qquad(x\geq89,387),\\
\vartheta(x)>&0.990x,\qquad(x\geq32,057),\\
\vartheta(x)>&0.985x,\qquad(x\geq11,927).\\
\end{align*}
\end{cor}
\begin{thm}[\cite{Rosser-1975}]
If $x\geq10^8$, then
\begin{align*}
|\psi(x)-x|<&0.0242269\frac{x}{\log x},\\
|\vartheta(x)-x|<&0.0242269\frac{x}{\log x}.\\
\end{align*}
\end{thm}
\begin{cor}[\cite{Rosser-1975}]
If $x\geq525,752$, then
$$
\vartheta(x)-x\leq\psi(x)-x<0.024,2334\frac{x}{\log x}.
$$
\end{cor}
\begin{cor}[\cite{Rosser-1975}]
We have
\begin{align*}
|\vartheta(x)-x|<&0.024,2334\frac{x}{\log x},\hspace{20pt}(x\geq758,699),\\
|\vartheta(x)-x|<&\frac{1}{40}\frac{x}{\log x},\hspace{61pt}(x\geq678,407).\\
\end{align*}
\end{cor}
\begin{thm}[\cite{Rosser-1975}]
If $x>1$, then
\begin{align*}
|\psi(x)-x|<&\eta_k\frac{x}{\log^k x},\\
|\vartheta(x)-x|<&\eta_k\frac{x}{\log^k x},\\
\end{align*}
where
$$
\eta_2=8.6853,\qquad \eta_3=11,762,\qquad \eta_4=1.8559\cdot10^7.
$$
\end{thm}
\begin{thm}
If $\varepsilon(x)$ is defined as (\ref{(3.19) R-S}), then
\begin{align*}
\vartheta(x)-x\leq\psi(x)-x<&x\varepsilon(x),\hspace{35pt}(x>0),\\
\psi(x)-x\geq\vartheta(x)-x>&-x\varepsilon(x),\qquad(x\geq39.4).\\
\end{align*}
\end{thm}
\subsection{Improved estimates for $\psi-\vartheta$}
In this section we give some results from  \cite{Pereira}  to approximate the difference $\psi-\vartheta$ in terms of $\psi$ in quite a simple form. As consequences we deduce some estimates for $\psi-\vartheta$.
\begin{thm}[\cite{Pereira}]
For every $x>0$ we have
$$
\psi(x)-\vartheta(x)\leq\psi(x^{1/2})+\psi(x^{1/3})+\psi(x^{1/5}),
$$
$$
\psi(x)-\vartheta(x)\geq\psi(x^{1/2})+\psi(x^{1/3})+\psi(x^{1/7}).
$$
\end{thm}
\begin{thm}[\cite{Pereira}]
We have
$$
\psi(x)-\vartheta(x)<\sqrt{x}+\frac43\sqrt[3]{x},\qquad(0<x\leq10^8),
$$
$$
\psi(x)-\vartheta(x)>\sqrt{x}+\frac23\sqrt[3]{x},\qquad(2187\leq x\leq10^8).
$$
\end{thm}
\begin{thm}[\cite{Pereira}]
$$
\psi(x)<x+0.656\sqrt{x}+\frac43\sqrt[3]{x},\qquad(0<x\leq10^8).
$$
For $1427\leq x\leq3298,\ 3299\leq x\leq19371$ or $19373\leq x\leq10^8$
$$
\psi(x)>x-0.833\sqrt{x}+\frac23\sqrt[3]{x}.
$$
\end{thm}
\begin{thm}[\cite{Pereira}]
$$
\psi(x)-\vartheta(x)<\sqrt{x}+\frac65\sqrt[3]{x},\qquad(10^8\leq x\leq10^{16}),
$$
$$
\psi(x)-\vartheta(x)>\sqrt{x}+\frac67\sqrt[3]{x},\qquad(10^8\leq x\leq10^{16}).
$$
\end{thm}
\begin{thm}[\cite{Pereira}]
With the aid of a computer, it can be easily verified that
$$
\psi(x)-\vartheta(x)<\sqrt{x}+\frac65\sqrt[3]{x},\qquad(8,236,167\leq x\leq10^{16}),
$$
$$
\psi(x)-\vartheta(x)>\sqrt{x}+\frac67\sqrt[3]{x},\qquad(2,036,329\leq x\leq10^{16}).
$$
\end{thm}
The results whivh have given so far are strictly elementary. However, in order to
estimate $\psi(x)-\vartheta(x)$ for $x > 10^{16}$, one needs the following bounds for $\psi$ which were
 deduced by Schoenfeld \cite{Sch-1976}, using powerful analytical methods.
\begin{align*}
|\psi(x)-x|<&0.00119721x,\qquad(10^{8}\leq x<e^{18.43}),\\
|\psi(x)-x|<&0.0011930x,\ \qquad(e^{18.43}\leq x<e^{18.44}),\\
|\psi(x)-x|<&0.0011885x,\ \qquad(e^{18.44}\leq x<e^{18.45}),\\
|\psi(x)-x|<&0.0011839x,\ \qquad(e^{18.45}\leq x<e^{18.46}),\\
|\psi(x)-x|<&0.0011615x,\ \qquad(e^{18.46}\leq x<e^{18.47}),\\
|\psi(x)-x|<&0.0010765x,\ \qquad(e^{18.7}\leq x<e^{19}),\\
|\psi(x)-x|<&0.00096161x,\qquad(x\geq e^{19}).\\
\end{align*}
\begin{thm}[\cite{Pereira}]
\begin{align*}
\psi(x)-\vartheta(x)<&1.001\sqrt{x}+1.1\sqrt[3]{x},\qquad(x\geq 10^{16}),\\
\psi(x)-\vartheta(x)<&1.001\sqrt{x}+\sqrt[3]{x},\quad\ \qquad(x\geq e^{38}),\\
\psi(x)-\vartheta(x)>&0.999\sqrt{x}+0.9\sqrt[3]{x},\qquad(x\geq 10^{16}),\\
\psi(x)-\vartheta(x)>&0.999\sqrt{x}+\sqrt[3]{x},\quad\ \qquad(x\geq e^{38}).\\
\end{align*}
\end{thm}
лллллллллллллллллллллллллл
For $x=10^8$, we have $\delta=2.44\cdot10^{-4}$, $m=2$, $\varepsilon=0.00118294$. For $x=10^{16}$, we have $\delta=5.24\cdot10^{-8}$, $m=2$, $\varepsilon=4.66629\cdot10^{-7}$.
\begin{thm}
We have
$$
\vartheta(x)<1.000027651x,\qquad(x>0),
$$
$$
\vartheta(x)>0.99871149x,\qquad(x\geq10^8).
$$
\end{thm}
\begin{proof}
If $8\cdot10^{11}\leq x<e^{28}$, then
\begin{align*}
\vartheta(x)<\psi(x)-\sqrt{x}-\frac67\sqrt[3]{x}<&\left\{1.0000284888-\frac1{\sqrt{x}}-\frac67\frac1{\sqrt[3]{x^2}}\right\}x\\
                                                <&\left\{1.0000284888-e^{-28/2}-\frac67e^{-28(2/3)}\right\}x\\
                                                <&1.000027651x.
\end{align*}
By handling the intervals $[e^{28}, e^{29})$, etc., similarly, we derive the same inequality. And for $x\geq e^{28}$ we use the table and $\vartheta(x)<\psi(x)$. This proves for all $x\geq8\cdot10^{11}$. For $x<8\cdot10^{11}$, it follows from (4.5) of \cite{Rosser-1962} and Dusart\cite{Dusart-3} which says $\vartheta(x)<x$ in this domain.\\

If $8\cdot10^{11}\leq x<10^{16}$, then
\begin{align*}
\vartheta(x)>\psi(x)-\sqrt{x}-\frac65\sqrt[3]{x}>&\left\{1-0.0000284888-\frac1{\sqrt{x}}-\frac65\frac1{\sqrt[3]{x^2}}\right\}x\\
                                                >&\left\{0.9999715112-(8\cdot10^{11})^{-1/2}-\frac65 (8\cdot10^{11})^{-(2/3)}\right\}x\\
                                                >&0.9999703792x.
\end{align*}
If $x\geq10^{16}$
\begin{align*}
\vartheta(x)>\psi(x)-1.001\sqrt{x}-1.1\sqrt[3]{x}>&\left\{1-4.66629\cdot10^{-7}-1.001\frac1{\sqrt{x}}-1.1\frac1{\sqrt[3]{x^2}}\right\}x\\
                                                >&\left\{0.999999533371-(1.001)10^{-8}-(1.1) 10^{-16(2/3)}\right\}x\\
                                                >&0.9999995233373x
\end{align*}
\end{proof}
\begin{thm}
If $x\geq8\cdot10^{11}$
$$
|\psi(x)-x|<0.000797686\frac{x}{\log x},
$$
$$
|\vartheta(x)-x|<0.000821232\frac{x}{\log x}.
$$
\end{thm}
\begin{proof}
If $8\cdot10^{11}\leq x<e^{28}$, then
\begin{align*}
\vartheta(x)-x>&\psi(x)-x-\sqrt{x}-\frac65\sqrt[3]{x}\\
              >&-\left\{\left(0.0000284888+\frac{1}{\sqrt{x}}+\frac65\frac{1}{\sqrt[3]{x^2}}\right)\log x\right\}\frac{x}{\log x}\\
>&-\left\{\left(0.0000284888+e^{-28/2}+\frac65e^{-28(2/3)}\right)(28)\right\}\frac{x}{\log x}\\
>&-0.000821232\frac{x}{\log x}.
\end{align*}
We continue to use the table in this way until $e^{35}$. If $10^{16}\leq x<e^{40}$
\begin{align*}
\vartheta(x)-x>&\psi(x)-x-1.001\sqrt{x}-1.1\sqrt[3]{x}\\
              >&-\left\{\left(4.66629\cdot10^{-7}+1.001\frac1{\sqrt{x}}+1.1\frac1{\sqrt[3]{x^2}}\right)\right\}\frac{x}{\log x}\\
                                                >&-\left\{\left(4.66629\cdot10^{-7}+(1.001)e^{-20}+(1.1)e^{-40(2/3)}\right)40\right\}x\\
                                                >&-0.00001874781\frac{x}{\log x}.
\end{align*}
We continue again until $e^{1000}$. For $x\geq e^{1000}$, we apply Theorem \ref{bound for psi x-x} and note that $\varepsilon(x)\log x<0.012559$, so that
$$
\vartheta(x)-x>-\{\varepsilon(x)\log x\}\frac{x}{\log x}>-0.012559\frac{x}{\log x}.
$$
\end{proof}
\begin{thm}
If $x>8\cdot10^{11}$, then
$$
|\psi(x)-x|<\eta_k\frac{x}{\log^k x},
$$
where
\begin{table}[h!]
  \centering
  \caption{$x>8\cdot10^{11}$}\label{eta1}
\begin{tabular}{|c|c|c|c|c|}
  \hline
  k & 1 &  2 & 3 & 4\\
\hline
&&&&\\
  $\eta_k$ & 0.000797686 & 0.0223352 & 0.625386 & 1230 \\
  \hline
\end{tabular}
\end{table}
\end{thm}

\begin{proof}
If $x\geq8\cdot10^{11}$ we proceed as previous theorem. For $x\geq8\cdot10^{11}$ from table we get. \\
If $1<x<8\cdot10^{11}$, since $\vartheta(x)<x$, we have
\begin{align*}
\psi(x)-x<&\vartheta(x)-x+\sqrt{x}+\frac43\sqrt[3]{x}<\left\{\frac{\log^k x}{\sqrt{x}}+\frac43\frac{\log^k x}{\sqrt[3]{x^2}}\right\}\frac{x}{\log^k x}\\
\leq&\left\{\frac{(2k)^k}{e^k}+\frac43\frac{(3k/2)^k}{e^k}\right\}\frac{x}{\log^k x},\qquad(k=1,2, 3, 4)
\end{align*}
For $k=1,2, 3, 4$; and
\begin{align*}
\psi(x)-x>&\vartheta(x)-x+\sqrt{x}+\frac23\sqrt[3]{x}>-2.06\sqrt{x}+\sqrt{x}+\frac23\sqrt[3]{x}\\
=&\left\{-1.06\frac{\log^k x}{\sqrt{x}}+\frac23\frac{\log^k x}{\sqrt[3]{x^2}}\right\}\frac{x}{\log^k x}\\
>&-c_k\frac{x}{\log^k x}
\end{align*}
where $c_1=0.445$ and $c_2=1.592$ and $c_3=8.887$ and $c_4=66.8894$.
\end{proof}

\begin{thm}
For $x\geq8\cdot10^{11}$
$$
|\vartheta(x)-x|<\eta_k\frac{x}{\log^k x}
$$
where
\begin{table}[h!]
  \centering
  \caption{$\eta$}\label{eta}
\begin{tabular}{|c|c|c|c|c|}
  \hline
  k & 1 & 2 & 3 & 4\\
\hline
&&&&\\
  $\eta_k$ & 0.000821232 & 0.0229945 & 0.643846 & 1230\\
  \hline
\end{tabular}
\end{table}

\end{thm}
\begin{proof}
For $1<x<10^8$
\begin{align*}
\vartheta(x)-x>&-2.06\sqrt{x}=-2.06\frac{\log^k x}{\sqrt{x}}.\frac{x}{\log^k x}\\
              \geq&-2.06\frac{(2k)^k}{e^k}\frac{x}{\log^k x}\\
\end{align*}
for $k=1, 2, 3, 4$.
\end{proof}
\begin{thm}
If $\varepsilon(x)$ is defined as in Theorem  \ref{bound for psi x-x}, then
$$
\vartheta(x)-x\leq\psi(x)-x<x\varepsilon(x),\qquad(x>0)
$$
$$
\psi(x)-x\geq\vartheta(x)-x>-x\varepsilon(x),\qquad(x\geq71)
$$
\end{thm}

\begin{proof}
We need to verify them for $x<e^{110}$. As $\varepsilon(x)$ increases for $1<x<12644$ and decreases for $x>12645$. We have $0.0357<\varepsilon(x)<0.2304221$ for $5\leq x<e^{110}$. From the table we deduce them for $10^8\leq x<e^{110}$. For $132\leq x<10^8$, we have $\varepsilon(x)>0.204$. Hence by (3.35) of Rosser 1961,
$$
\psi(x)<1.04x<(1+\varepsilon(x))x
$$
For $1<x<132$, we use direct computation.

From Rosser 1961, For $110\leq x<10^8$
$$
\vartheta(x)>0.84x>(1-\varepsilon(x))x
$$
For $71\leq x<110$, we use direct computation.

\textbf{Second method.} By Theorem 9 of \cite{Rosser-1975}
$$
\psi(x)-x<x\varepsilon_3(x),\qquad(x>0)
$$
where $\varepsilon_3$ is defined in (3.9) of \cite{Rosser-1975}. On the other hand $\varepsilon_3(x)<\varepsilon(x)$ for $408<x<e^{190}$. By computation for smaller values. The same hold for $\vartheta(x)-x>-x\varepsilon(x)$.
\end{proof}
\subsection{Sharper bounds for $|\psi(x)-x|$ and $|\vartheta(x)-x|$}
\begin{lem}[\cite{Rosser-1975}]\label{lemma 4 R-S}
If $\nu\leq1$, $z>0$, and $x>1$, we have
$$
K_\nu(z,x)<Q_\nu(z,x)
$$
where
$$
Q_\nu(z,x)=\frac{x^{\nu+1}}{z(x^2-1)}H^z(x),\qquad H(t)=e^{-\frac12(t+1/t)}
$$
\end{lem}
\begin{lem}[\cite{Rosser-1975}]
If $z>0$ and $x>0$, then
\begin{align*}
(x-1)Q_1(z,x)+&(1+\frac2z-\frac2{z(1+x)^2})K_1(z,x)\\
             <&K_2(z,x)<(x-1)Q_1(z,x)+(1+\frac2z)K_1(z,x)
\end{align*}
\end{lem}
\begin{cor}[\cite{Rosser-1975}]\label{lemma 5, corollary R-S}
If $z>0$ and $x>1$, then
$$
K_2(z,x)<(x+\frac2z)Q_1(z,x)
$$
\end{cor}
\begin{thm}\label{8/pi}
Let
$$
\varepsilon_0(x)=\sqrt{8/\pi} X^{1/2}e^{-X}
$$
Then
$$
|\psi(x)-x|<x\varepsilon_0(x),\qquad(x\geq3)
$$
and
$$
|\vartheta(x)-x|<x\varepsilon_0(x),\qquad(x\geq3)
$$
\end{thm}
\begin{proof}
The main part of the proof is concerned with large $x$ in which case the proof is similar to Theorem \ref{theorem 3 R-S}, but we ultimately take $m=2$ rather than $m=1$. In place of (\ref{(3.36) R-S}), we let

\begin{equation}\label{(7.5) R-S}
T_2= e^{\nu x},
\end{equation}
where $\nu$ will be specified later. We assume that $\nu$, $m$, $X$ are such that

\begin{equation}\label{(7.6) R-S}
T_2\geq A,\qquad \frac{1}{\sqrt{m+1}}\leq\nu\leq1
\end{equation}
from which we deduce
$$
X\geq\nu X=\log T_2 \geq\log A
$$
and
$$
W_m= e^{X/\sqrt{m+1}}\leq T_2= e^{\nu X}\leq  e^X=W_0.
$$
In place of (\ref{(3.37) R-S}), we get
\begin{equation}\label{7.7}
S_3(m,\delta)\leq\frac{2+m\delta}{2}\left(\left\{\frac{1}{2\pi}-q(T_2)\right\}\int_A^{T_2}\phi_0(y)\log\frac{y}{2\pi}dy+E_1\right)
\end{equation}
where
\begin{align}
E_1=&\{N(T_2)-F(T_2)+R(T_2)\}+\phi_0(T_2)-\{N(A)-F(A)+R(A)\}\phi_0(A)\nonumber\\
   <&2R(T_2)\phi_0(T_2)\label{(7.8) R-S}
\end{align}
and $R(T)=0.137\log T+0,443\log\log T+1.588$. Putting
$$
V''=\frac{X^2}{\log T_2}
$$
we have
\begin{equation}\label{(7.9) R-S}
V''=\frac{X^2}{\nu X}=\frac{X}{\nu}=X\{\frac{(1-\nu)^2}{\nu}+2-\nu\}=Y+2X-\nu X,
\end{equation}
where
\begin{equation}\label{(7.10) R-S}
Y=X\frac{(1-\nu)^2}{\nu}
\end{equation}
Proceeding as in (\ref{(3.38) R-S}) and (\ref{(3.41) R-S}) and using (\ref{(7.8) R-S}), we find
\begin{align}\label{(7.11) R-S}
S_3(m,\delta)<&\frac{2+m\delta}{2}\frac{1}{2\pi}e^{-V''}\{X^4(V'')^{-3}-(\log2\pi)X^2(V'')^{-2}\}+\frac{2+m\delta}{2}E_1\nonumber\\
             <&\frac{2+m\delta}{4\pi}e^{-Y-2X}T_2XG_0+(2+m\delta)R(T_2)\phi_0(T_2)
\end{align}
where
\begin{equation}\label{(7.12) R-S}
G_0=\nu^2\left(\nu-\frac{\log2\pi}{X}\right)
\end{equation}
As $R(y)/\log y$ decrease for $y>e^e$, we have
\begin{align*}
R(T_2)\phi_0(T_2)=&\frac{R(T_2)}{\log T_2}\phi_0(T_2)\log T_2\leq\frac{R(A)}{\log A}\phi_0(T_2)\log T_2\\
              \leq&\frac{R(A)}{\log A}\frac{e^{-V''}}{T_2}\log T_2=\frac{R(A)}{\log A}\frac{e^{-Y-2X+\nu X}}{T_2}\log T_2\\
              =&\frac{R(A)}{\log A}e^{-Y-2X}\log T_2.\\
\end{align*}
by (\ref{(7.9) R-S}). Then (\ref{(7.5) R-S}) and (\ref{(7.6) R-S}) yield
$$
\log T_2=\nu X\leq X
$$
hence
\begin{equation}\label{(7.13) R-S}
R(T_2)\phi_0(T_2)<0.24471Xe^{-Y-2X}.
\end{equation}
We have
\begin{equation}\label{7.14}
S_4(m,\delta)\leq\frac{R_m(\delta)}{\delta^m}\left(\left\{\frac{1}{2\pi}+q(T_2)\right\}\int_{T_2}^\infty\phi_m(y)\log\frac{y}{2\pi}dy+E_0\right)
\end{equation}
where
\begin{align}
E_0=&\{R(T_2)+F(T_2)-N(T_2)\}\phi_m(T_2)\nonumber\\
   <&2R(T_2)\phi_m(T_2)=2R(T_2)\phi_0(T_2)T^{-m}\label{7.15}
\end{align}
Also
\begin{equation}\label{7.16}
\int_{T_2}^\infty\phi_m(y)\log\frac{y}{2\pi}dy=\frac{z^2}{2m^2}\{K_2(z,U')-\frac{2m\log2\pi}{z}K_1(z,U')\}
\end{equation}
where $z=2X\sqrt{m}$ and
$$
U'=\frac{2m}{z}\log T_2=\frac{2m}{2X\sqrt{m}}\log T_2=\sqrt{m}\frac{\log T_2}{X}=\sqrt{m}\frac{\nu X}{X}=\nu\sqrt{m}
$$
By assuming
\begin{equation}\label{(7.17) R-S}
\nu>\frac{1}{\sqrt{m}}
\end{equation}
we have $U'>1$; also $m\geq2$ since $\nu\leq1$.

By Lemma \ref{lemma 4 R-S} and Corollary \ref{lemma 5, corollary R-S}
\begin{align*}
K_2(z,U')-\frac{2m\log2\pi}{z}K_1(z,U')<&K_2(z,U')<\left(U'+\frac2z\right)Q_1(z,U')\\
                     =&\sqrt{m}\left(\nu+\frac{1}{mX}\right)\frac{U'^2}{z(U'^2-1)}e^{-\frac12z(U'+1/U')}
\end{align*}
Now
$$
\frac12z\left(U'+\frac{1}{U'}\right)=X\sqrt{m}\left(\nu\sqrt{m}+\frac{1}{\nu\sqrt{m}}\right)=m\nu X+(Y+2X-\nu X)
$$
Hence,
\begin{align}
K_2(z,U')<&\sqrt{m}\left(\nu+\frac{1}{mX}\right)\frac{m\nu^2}{2X\sqrt{m}(m\nu^2-1)}e^{-m\nu X-(Y+2X-\nu X)}\nonumber\\
         =&\left(\nu+\frac{1}{mX}\right)\frac{m}{2(m-1)}\frac{(m-1)\nu^2}{(m\nu^2-1)}X^{-1}T_2^{-(m-1)}e^{-Y-2X}\nonumber\\
         =&G_1\frac{m}{2(m-1)}X^{-1}T_2^{-(m-1)}e^{-Y-2X}\label{7.18}
\end{align}
where

\begin{equation}\label{7.19}
G_1=\frac{(m-1)\nu^2}{(m\nu^2-1)}\left(\nu+\frac{1}{mX}\right)
\end{equation}
Then
\begin{align*}
\int_{T_2}^\infty\phi_m(y)\log\frac{y}{2\pi}dy<&\frac{z^2}{2m^2}G_1\frac{m}{2(m-1)}X^{-1}T_2^{-(m-1)}e^{-Y-2X}\\
                                              =&\frac{1}{m-1}G_1XT_2^{-(m-1)}e^{-Y-2X}
\end{align*}
We define
\begin{equation}\label{7.20}
G_2=\frac{R_m(\delta)}{2^m}\{1+2\pi q(T_2)\}=\{1+2\pi q(T_2)\}\left\{\frac{(1+\delta)^{m+1}+1}{2}\right\}^m
\end{equation}
Then
\begin{align*}
S_4(m,\delta)\leq&\frac{R_m(\delta)}{2^m}\left(\frac{2}{\delta}\right)^m\frac{1}{2\pi}\{1+2\pi q(T_2)\}\int_{T_2}^\infty\phi_m(y)\log\frac{y}{2\pi}dy+\frac{R_m(\delta)}{2^m}\left(\frac{2}{\delta}\right)^m E_0\\
=&\left(\frac{2}{\delta}\right)^m\frac{1}{2\pi}G_2\int_{T_2}^\infty\phi_m(y)\log\frac{y}{2\pi}dy+\frac{R_m(\delta)}{2^m}\left(\frac{2}{\delta}\right)^m E_0\\
<&\left(\frac{2}{\delta}\right)^m\frac{1}{2\pi}G_2\int_{T_2}^\infty\phi_m(y)\log\frac{y}{2\pi}dy+G_2\left(\frac{2}{\delta}\right)^m E_0\\
<&\left(\frac{2}{\delta}\right)^m\frac{1}{2\pi(m-1)}G_2G_1XT_2^{-(m-1)}e^{-Y-2X}+G_2\left(\frac{2}{\delta}\right)^m E_0\\
\end{align*}
Now $1+m\delta<R_m(\delta)/2^m<G_2$. We obtain
\begin{align*}
S_3(m,\delta)+S_4(m,\delta)<&\frac{1}{2\pi}G_2Xe^{-Y-2X}\left\{G_0T_2+\frac{1}{m-1}G_1\left(\frac{2}{\delta T_2}\right)^mT_2\right\}\\
&+2G_2R(T_2)\phi_0(T_2)\left\{1+\left(\frac{2}{\delta T_2}\right)^m\right\}
\end{align*}
If $G_0$ and $G_1$ were independent of $\nu$, and hence of $T_2$, then the expression inside the first braces would be minimized by choosing
\begin{equation}\label{7.21}
T_2=\frac2\delta\left(\frac{G_1}{G_0}\right)^{1/m}
\end{equation}
Postponing the reconciliation of this with the previous definition of $T_2$, we obtain
\begin{align*}
S_3(m,\delta)+S_4(m,\delta)+\frac12m\delta<&\frac{1}{2\pi}G_2Xe^{-Y-2X}\left\{G_2T_2+\frac{1}{m-1}G_1\left(\frac{2}{\delta}\right)^mT_2^{1-m}\right\}\\
&+2G_2R(T_2)\phi_0(T_2)\left\{1+\left(\frac{2}{\delta T_2}\right)^m\right\}+\frac12mG_2\delta\\
=&\frac{1}{2\pi}G_2Xe^{-Y-2X}\left\{\frac{m}{m-1}G_0^{1-1/m}G_1^{1/m}\frac2\delta\right\}\\
&+2G_2R(T_2)\phi_0(T_2)\left(1+\frac{G_0}{G_1}\right)+\frac12mG_2\delta\\
=&\frac12mG_2\left\{\frac{2}{\pi(m-1)}G_0^{1-1/m}G_1^{1/m}Xe^{-Y-2X}\frac1\delta+\delta\right\}\\
&+2G_2R(T_2)\phi_0(T_2)\left(1+\frac{G_0}{G_1}\right)\\
\end{align*}
The expression inside the last braces is minimized by choosing
\begin{equation}\label{7.22}
\delta=\left\{\frac{2}{\pi(m-1)}G_0^{1-1/m}G_1^{1/m}e^{-Y}\right\}^{1/2}X^{1/2}e^{-X}
\end{equation}
so that (\ref{7.21}) becomes
\begin{equation}\label{7.23}
T_2=\frac2\delta\left(\frac{G_1}{G_0}\right)^{1/m}=\left(\frac{G_1}{G_0}\right)^{1/2m}\left\{\frac{2\pi(m-1)}{G_0}e^Y\right\}^{1/2}X^{-1/2}e^{X}
\end{equation}
Moreover, since $R(T_2)\phi_0(T_2)<0.24471Xe^{-Y-2X}$,
\begin{align*}
S_3(m,\delta)+S_4(m,\delta)+\frac12m\delta<&G_2\left\{\frac{2}{\pi}G_0^{1-1/m}G_1^{1/m}e^{-Y}\right\}^{1/2}\frac{m}{\sqrt{m-1}}X^{1/2}e^{-X}\\
&+0.48942G_2\left(1+\frac{G_0}{G_1}\right)e^{-Y}Xe^{-2X}\\
\end{align*}
The coefficient $\frac{m}{\sqrt{m-1}}$ in the next to the last term is minimized by choosing $m=2$. For this value we obtain

\begin{align}
\delta=&(G_0G_1)^{1/4}\sqrt{\frac2\pi}e^{-Y/2}X^{1/2}e^{-X}\label{7.24a}\\
T_2=&\left(\frac{G_1}{G_0^3}\right)^{1/4}\sqrt{2\pi}e^{Y/2}X^{-1/2}e^{X}\label{7.24b}\\
G_1=&\frac{\nu^2}{2\nu^2-1}\left(\nu+\frac{1}{2X}\right)\label{7.25}
\end{align}
Also
\begin{align}
S_3(2,\delta)+S_4(2,\delta)+\delta<&G_2(G_0G_1)^{1/4}e^{-Y/2}\sqrt{\frac{8}{\pi}}X^{1/2}e^{-X}\nonumber\\
&+0.48942G_2\left(1+\frac{G_0}{G_1}\right)e^{-Y}Xe^{-2X}\label{7.26}
\end{align}
provided the choice of $T_2$ in (\ref{7.24b}) is consistent with (\ref{(7.5) R-S}) and provided both (\ref{(7.6) R-S}) and (\ref{(7.17) R-S}) hold when $m=2$.

We readily see that $T_2$ of $T_2=e^{\nu X}$ and $T_2$ just above are equal if and only if $\nu$ is such that $k(\nu)=1$ where

\begin{equation}\label{7.27}
k(\nu)=\frac{1}{2\pi}X\left(\frac{G_0^3}{G_1}\right)^{1/2}e^{-2X(1-\nu)}e^{-X(1-\nu)^2/\nu}
\end{equation}

If $1/\sqrt{2}<\nu\leq\sqrt{3/2}$, it is not hard to see that $G_1$ decreases as $\nu$ increases, $G_0$ is also increasing function of $\nu$. Hence, $k(\nu)$ is strictly increasing for increasing $\nu\in(1/\sqrt{2},1]$. Now $k(\nu)\ra0$ as $\nu\ra1/\sqrt{2}$ from the right; and we easily see that $k(1)>1$ (for all $X\geq10$). As a result there is a unique $\nu\in(1/\sqrt{2},1)$ such that $k(\nu)=1$. Henceforth, let $\nu$ be this number so that $\nu$ depends on $X$; then $G_0$, $G_1$, $Y$ and $T_2$ are defined in terms of $\nu$ by (7.12), (7.25), (7.10) and (7.5), (7.24b). Of course, (7.17) holds since $m=2$. Hence (7.26) will be fully established once it is shown that $T\geq A$.
We have, for $1/\sqrt{2}<\nu\leq1$,

\begin{equation}\label{7.28}
H(\nu)\equiv\frac{G_0^3}{G_1}=\nu^4(2\nu^2-1)\frac{(\nu-\log(2\pi)/X)^3}{\nu+1/(2X)}\left\{
                                                                                  \begin{array}{ll}
                                                                                    <(\nu-\log(2\pi)/X)^2, & \hbox{all X;} \\\\
                                                                                    >0.37\nu^6(2\nu^2-1), & \hbox{$X\geq10$.}
                                                                                  \end{array}
                                                                                \right.
\end{equation}

If we define for $j=0$ and 1,

\begin{equation}\label{7.29}
\nu_j=1-\frac{1}{2X}\log\frac{X}{(2+3j)\pi}
\end{equation}
we see that
$$
H(\nu_0)<1,\quad(X\geq\frac{1}{2\pi}),\qquad\qquad H(\nu_1)>0.22318,\quad(X>0)
$$
Inasmuch as
\begin{align*}
k(\nu_j)=&\frac{1}{2\pi}XH(\nu_j)^{1/2}e^{-2X(1-\nu)}e^{-X(1-\nu)^2/\nu}\\
        =&\frac{1}{2\pi}XH(\nu_j)^{1/2}\exp\left\{-\log\frac{X}{(2+3j)\pi}\right\}\exp\left\{-\frac{1}{4\nu_j X}\log^2\frac{X}{(2+3j)\pi}\right\}\\
        =&\frac{2+3j}{2}H(\nu_j)^{1/2}\exp\left\{-\frac{1}{4\nu_j X}\log^2\frac{X}{(2+3j)\pi}\right\}\\
\end{align*}
we see that
$$
k(\nu_0)<1=k(\nu),\quad(X\geq\frac{1}{2\pi}),\qquad\qquad k(\nu_1)>1=k(\nu),\quad(X>0)
$$
So
\begin{equation}\label{7.30}
\nu_0<\nu,\quad(\log x\geq0.145),\qquad\qquad \nu<\nu_1,\quad(\log x>0)
\end{equation}
Of course $\nu_0<\nu_1$ in all cases. For $\log x\geq4890$, we get
$$
T_2>e^{\nu_0 X}>A
$$
Hence,
\begin{align*}
S_3(2,\delta)+S_4(2,\delta)+\delta<&G_2(G_0G_1)^{1/4}e^{-Y/2}\sqrt{\frac{8}{\pi}}X^{1/2}e^{-X}\\
&+0.48942G_2\left(1+\frac{G_0}{G_1}\right)e^{-Y}Xe^{-2X}\\
\end{align*}
for $\log x\geq4890$; and for these $x$, we have $\nu_0>0.9737>\sqrt{15/16}$. It is a simple matter to verify that
\begin{equation}\label{7.31}
\frac{G_0}{G_1}<2\nu^2-1<1,\qquad\qquad Y<X(1-\nu_0)^2/\nu_0<0.025
\end{equation}
\begin{equation}\label{7.32}
G_0G_1=\frac{\nu^6}{2\nu^2-1}\left(1-\frac{\log2\pi}{\nu X}\right)\left(1+\frac{1}{2\nu X}\right)\left\{
                                                                                  \begin{array}{ll}
                                                                                    <(1-\frac{3}{2\nu X})(1+\frac{1}{2\nu X})<1, & \hbox{$X\geq11$;} \\\\
                                                                                    >\frac12\nu^6/(2\nu^2-1)\geq0.84375, & \hbox{$X\geq11$.}
                                                                                  \end{array}
                                                                                \right.
\end{equation}
Then
\begin{align}
S_3(2,\delta)+S_4(2,\delta)+\delta<&G_2(G_0G_1)^{1/4}e^{-Y/2}\sqrt{\frac{8}{\pi}}X^{1/2}e^{-X}\nonumber\\
&+0.48942G_2\left(1+\frac{G_0}{G_1}\right)e^{-Y}Xe^{-2X}\nonumber\\
<&G_2(G_0G_1)^{1/4}e^{-Y/2}\left\{\sqrt{\frac{8}{\pi}}X^{1/2}e^{-X}+1.022Xe^{-2X}\right\}\label{7.33}
\end{align}
Taking $T_1=0$ and using Proposition \ref{sum 1/gamma2,3,4,5,6,7}, we obtain
\begin{align*}
\frac{1}{\sqrt{x}}\{S_1(2,\delta)+S_2(2,\delta)\}<&\frac{1}{\sqrt{x}}\frac{R_2(\delta)}{\delta^2}\sum_\gamma\frac{1}{\gamma^3}<\frac{1}{\sqrt{x}}G_2\left(\frac2\delta\right)^2(0.00146435)\\
<&\frac{4}{\sqrt{x}}G_2(0.00146435)(G_0G_1)^{-1/2}\frac\pi2 e^{Y}X^{-1}e^{2X}\\
<&0.009434\frac{1}{\sqrt{x}}G_2(G_0G_1)^{-1/2}X^{-1}e^{2X}\\
<&0.010715\frac{1}{\sqrt{x}}G_2(G_0G_1)^{1/4}X^{-1}e^{2X}
\end{align*}
Putting
$$
\Omega=\frac{1}{\sqrt{x}}\{S_1(2,\delta)+S_2(2,\delta)\}+S_3(2,\delta)+S_4(2,\delta)+\delta
$$
we obtain  that for $\log x\geq4890$
\begin{align}\label{7.34}
\Omega<G_2(G_0G_1)^{1/4}e^{-Y/2}\left\{\sqrt{\frac{8}{\pi}}X^{1/2}e^{-X}+1.022Xe^{-2X}+0.010715\frac{1}{\sqrt{x}}X^{-1}e^{2X}\right\}
\end{align}
Since
$$
\frac1x|\psi(x)-x|<\frac1x\left\{\log2\pi+\frac12\log(1-1/x^2)\right\}+\Omega<\Omega+\frac{\log2\pi}x
$$
Now
\begin{align*}
\frac1x|\vartheta(x)-x|<&\Omega+\frac{\log2\pi}x+\frac{1.43}{\sqrt{x}}\\
<&\Omega+\frac{0.000085}{\sqrt{x}}G_2(G_0G_1)^{1/4}e^{-Y/2}X^{-1}e^{2X}
\end{align*}
Hence,
\begin{equation}\label{7.35}
\frac1x|\psi(x)-x|,\qquad\frac1x|\vartheta(x)-x|<G_3\sqrt{\frac8\pi}X^{1/2}e^{-X}=G_3\varepsilon_0(x)
\end{equation}
for $\log x\geq4890$, where
\begin{align}
G_3=&G_2(G_0G_1)^{1/4}e^{-Y/2}\left\{1+\sqrt{\frac\pi8}(1.022X^{1/2}e^{-X}+\frac{0.0108}{\sqrt{x}}X^{-3/2}e^{3X})\right\}\nonumber\\
<&G_2(G_0G_1)^{1/4}e^{-Y/2}\left\{1+0.65X^{1/2}e^{-X}\right\}\nonumber\\
<&G_2(G_0G_1)^{1/4}e^{-Y/2}\exp\left\{0.65X^{1/2}e^{-X}\right\}\label{7.36}
\end{align}
Also by definition of $q(y)$, relation for $T_2$, and $H(\nu)$
\begin{align*}
1+2\pi q(T_2)=&1+2\pi\frac{0.137+0.443/\log T_2}{T_2\log(T_2/2\pi)}\\
<&1+\frac{2\pi}{T_2}\frac{0.137+0.443/\log A}{\log(A/2\pi)}\\
<&1+0.0057154\frac{2\pi}{T_2}\\
=&1+0.0057154\sqrt{2\pi}\left(\frac{G_0^3}{G_1}\right)^{1/4}e^{-Y/2}X^{1/2}e^{-X}\\
<&1+0.0143265X^{1/2}e^{-X}\\
<&\exp(0.0143265X^{1/2}e^{-X})
\end{align*}
Further,
\begin{align*}
\frac{R_2(\delta)}{2^2}=&\left\{\frac{(1+\delta)^3+1}{2}\right\}^2=\left\{1+\frac12\delta(3+3\delta+\delta^2)\right\}^2<\left(1+\frac{3.01}{2}\delta\right)^2\\
<&\left\{\exp\left(\frac{3.01}{2}\delta\right)\right\}^2=\exp(3.01\delta)<\exp(2.402X^{1/2}e^{-X})
\end{align*}
Then
\begin{align*}
G_3<&G_2(G_0G_1)^{1/4}e^{-Y/2}\exp\left\{0.65X^{1/2}e^{-X}\right\}\\
<&\exp(2.402X^{1/2}e^{-X})\exp(0.0143265X^{1/2}e^{-X})(G_0G_1)^{1/4}e^{-Y/2}\exp\left\{0.65X^{1/2}e^{-X}\right\}\\
<&(G_0G_1)^{1/4}e^{-Y/2}\exp(3.67X^{1/2}e^{-X})=\left\{G_0G_1e^{-2Y}\exp(14.68X^{1/2}e^{-X})\right\}^{1/4}
\end{align*}
By (\ref{(7.12) R-S}), we obtain for $\log x\geq4890$
\begin{align*}
\frac{X}{\nu^2}G_0\exp(14.68X^{1/2}e^{-X})=&(\nu X-\log2\pi)\exp(14.68X^{1/2}e^{-X})\\
<&(\nu X-\log2\pi)(1+14.69X^{1/2}e^{-X})\\
=&\nu X-\log2\pi+(\nu X-\log2\pi)14.69X^{1/2}e^{-X}\\
<&\nu X-\log2\pi+5\cdot10^{-10}<\nu X-1=X(\nu-1/X)
\end{align*}
Hence, (\ref{7.25}) yields
\begin{align*}
G_3<&\left\{G_0G_1e^{-2Y}\exp(14.68X^{1/2}e^{-X})\right\}^{1/4}<\left\{\nu^2(\nu-\frac1X)G_1e^{-2Y}\right\}^{1/4}\\
=&\left\{\frac{\nu^4}{2\nu^2-1}(\nu-\frac1X)(\nu+\frac{1}{2X})e^{-2Y}\right\}^{1/4}\\
\end{align*}
As a result of (\ref{7.35}) one deduce for $\log x\geq4890$

\begin{equation}\label{7.37}
|\psi(x)-x|,\qquad|\vartheta(x)-x|<x\varepsilon_0(x)M(\nu)L(\nu)
\end{equation}
where
\begin{equation}\label{7.38}
L(\nu)=\left\{\frac{\nu^6}{2\nu^2-1}\right\}^{1/4}
\end{equation}
\begin{equation}\label{7.39}
M(\nu)=\left\{(1-\frac{1}{\nu X})(1+\frac{1}{2\nu X})e^{-2X(1-\nu)^2/\nu}\right\}^{1/4}
\end{equation}
The function $L(\nu)$ is real valued for $\nu>1/\sqrt{2}$ and, as is easily seen, has a minimum value at $\nu=\sqrt{3/4}$. If $\log x\geq0.145$, then $\nu>\nu_0>1/\sqrt{2}$. Hence,
\begin{equation}\label{7.40}
\left(\frac{27}{32}\right)^{1/4}<L(\nu)<1,\qquad(x\geq e^{0.145})
\end{equation}
In addition,
\begin{align}
M(\nu)<&\exp\frac14\left\{-\frac{1}{\nu X}+\frac{1}{2\nu X}-\frac{2X}{\nu}(1-\nu)^2\right\}\nonumber\\
      <&\exp\frac14\left\{-\frac{1}{\nu X}+\frac{1}{2\nu X}-\frac{2X}{\nu}(1-\nu_1)^2\right\}\nonumber\\
      =&\exp\frac14\left\{-\frac{1}{2\nu X}\left(1+\log^2\frac{X}{5\pi}\right)\right\}<E(x)\label{7.41}
\end{align}
where

\begin{equation}\label{7.42}
E(x)=\exp\frac14\left\{-\frac{1}{2\nu_1 X}\left(1+\log^2\frac{X}{5\pi}\right)\right\}=\exp\frac1{4\nu_1}\left\{-\frac{1}{2 X}-2X(1-\nu_1)^2\right\}
\end{equation}

It is clear from the first part of (\ref{7.42}) that $E(x)<1$ for all $x$. So for $\log x\geq4890$
$$
|\psi(x)-x|,\qquad|\vartheta(x)-x|<x\varepsilon_0(x).
$$
For smaller $x$ we use the table and relation $0\leq\psi(x)-\vartheta(x)<1.427\sqrt{x}$.
\end{proof}

\subsubsection{Numerical Bounds for Moderate Values of $x$}
Let
$$
T_0=\frac1\delta\left(\frac{2R_m(\delta)}{2+m\delta}\right)^{1/m}
$$
and leave $T_1$ unspecified for the moment. We showed that by letting $T_2=e^{\nu X}$,
\begin{align*}
S_3(m,\delta)<\frac{2+m\delta}{2}\left(\left\{\frac{1}{2\pi}-q(T_2)\right\}\int_A^{T_2}\phi_0(y)\log\frac{y}{2\pi}dy+2R(T_2)\phi_0(T_2)\right)
\end{align*}
\begin{align*}
S_4(m,\delta)<\frac{R_m(\delta)}{\delta^m}\left(\left\{\frac{1}{2\pi}+q(T_2)\right\}\int_{T_2}^\infty\phi_m(y)\log\frac{y}{2\pi}dy+2R(T_2)\phi_0(T_2)T^{-m}\right)
\end{align*}
so that
\begin{align*}
S_3(m,\delta)+S_4(m,\delta)<\frac{1}{2\pi}h_3(T_2)+e_3(T_2)
\end{align*}
where
$$
h_3(T)=\frac{2+m\delta}{2}\int_A^{T}\phi_0(y)\log\frac{y}{2\pi}dy+\frac{R_m(\delta)}{\delta^m}\int_{T}^\infty\phi_m(y)\log\frac{y}{2\pi}dy
$$
and
\begin{align*}
e_3(T)=&q(T)\left\{-\frac{2+m\delta}{2}\int_A^{T}\phi_0(y)\log\frac{y}{2\pi}dy+\frac{R_m(\delta)}{\delta^m}\int_{T}^\infty\phi_m(y)\log\frac{y}{2\pi}dy\right\}\\
&+R(T)\phi_0(T)\{2+m\delta+2\frac{R_m(\delta)}{(\delta T)^m}\}
\end{align*}
The situation for $S_1(m,\delta)+S_2(m,\delta)$ is entirely similar. If we leave $T_1$ and $D$ unspecified but subject to $2\leq D\leq A$ and $T_1\geq D$, then we get
$$
S_1(m,\delta)+S_2(m,\delta)<\frac{1}{\pi}h_1(T_1)+e_1(T_1)
$$
where
\begin{align*}
h_1(T)=&\frac{2+m\delta}{2}\int_D^T\frac1y\log\frac{y}{2\pi}dy+\frac{R_m(\delta)}{\delta^m}\int_T^\infty\frac{1}{y^{m+1}}\log\frac{y}{2\pi}dy\\
&+(2+m\delta)\pi\left\{G(D)+\frac{1}{4\pi}\log^2\frac{D}{2\pi}\right\}
\end{align*}
\begin{align*}
G(D)=&\sum_{0<\gamma\leq D}\frac{1}{(\gamma^2+1/4)^{1/2}}-\frac{1}{4\pi}\left\{\left(\log\frac{D}{2\pi}-1\right)^2+1\right\}\\
&+\frac1D\left\{0.137\log D+0.443\left(\log\log D+\frac{1}{\log D}\right)+2.6-N(D)\right\}
\end{align*}
\begin{align*}
e_1(T)=&-\frac{2\pi}{T}\left\{\frac{2+m\delta}{2}\left(0.137+\frac{0.443}{\log D}\right)-\frac{R_m(\delta)}{(m+1)(\delta T)^m}\left(0.137+\frac{0.443}{\log T}\right)\right\}\\
&+\frac{2\pi}{T}\left\{\frac{2+m\delta}{2}-\frac{R_m(\delta)}{(\delta T)^m}\right\}\{N(T)-F(T)-R(T)\}
\end{align*}
Let
$$
C(D)=4\pi\left(0.137+\frac{0.443}{\log D}\right),\qquad S(D)=\sum_{0<\gamma\leq D}\frac{1}{(\gamma^2+1/4)^{1/2}}
$$
\begin{thm}
Let $T_0$ be defined as above and satisfy $T_0\geq D$, where $2\leq \leq A$. Let $m$ be a positive integer and let $\delta>0$. Then
$$
S_1(m,\delta)+S_2(m,\delta)<\Omega_1^\ast
$$
where
$$
\Omega_1^\ast=\frac{2+m\delta}{4\pi}\left\{\left(\log\frac{T_0}{2\pi}+\frac1m\right)^2+4\pi G(D)+\frac{1}{m^2}-\frac{m C(D)}{(m+1)T_0}\right\}
$$
where $G(D)$ and $C(D)$ are defined above. Moreover, if
$$
\Omega_3^\ast=\frac{1}{2\pi}h_3(T_2)+e_3(T_2)
$$
then
$$
|\psi(x)-x|<\varepsilon_0^\ast x,\qquad(x\geq e^b)
$$
where
$$
\varepsilon_0^\ast=\Omega_1^\ast e^{-b/2}+\Omega_3^\ast+\frac m2\delta+e^{-b}\log2\pi
$$
\end{thm}
\begin{rem}
After doing this article we realized that a similar theorem to Theorem \ref{8/pi} for $|\psi(x)-x|$ was done by P. Dusart. But as you may see in this article we give the proofs with more details.
\end{rem}

\section*{Acknowledgments}
The work of the second author is supported by Center of Mathematics of the University of Porto.
The work of the first author is supported by the Calouste Gulbenkian Foundation, under
Ph.D. grant number CB/C02/2009/32. Research partially funded by the European Regional
Development Fund through the programme COMPETE and by the Portuguese Government
through the FCT under the project PEst-C/MAT/UI0144/2011. Sincere thanks to \'{E}lio Coutinho from Informatics Center of Faculty of Science of the University of Porto for providing accesses to fast computers for necessary computations.

\bibliography{bibliojab1}
\bibliographystyle{plain}

\vspace*{10mm}
\noindent Email: sdnazdi@yahoo.com\\\\
Email: syakubov@fc.up.pt\\

\newpage
in this case for $b\geq 5765$
$$
\varepsilon<\Omega_1e^{-b/2}+\Omega_2+md/2+\log(2\pi)e^{-b}
$$
\begin{table}
\vspace*{-27mm}
\small
  \centering
  \caption{$|\psi(x)-x|<x\varepsilon,\quad(x\geq e^b)$, $\varepsilon=C X^{3/4}e^{-X}$}\label{d e}
\begin{tabular}{|c|c|c|c|c|c|c|c|c|}
  \hline
  $b$&$m$&$\delta$&$\varepsilon$&&$b$&$m$&$\delta$&$\varepsilon$ \\
 \hline 
  18.42 & 1 & 4.77(-4) & 1.14853(-3) && 950 & 21 &  2.15(-12) & 2.36304(-11) \\
  18.43 & 1 & 4.75(-4) & 1.14399(-3) && 1000 & 21 &  2.12(-12) & 2.32993(-11) \\
  18.44 & 1 & 4.73(-4) & 1.13947(-3) && 1050 & 21 &  2.09(-12) & 2.29730(-11) \\
  18.45 & 1 & 4.71(-4) & 1.13496(-3) && 1100 & 20 &  2.16(-12) & 2.26446(-11) \\
  18.5 & 1 &  4.61(-4) & 1.11269(-3) && 1150 & 20 &  2.13(-12) & 2.23136(-11) \\
  18.7 & 1 &  4.22(-4) & 1.02778(-3) && 1200 & 20 &  2.09(-12) & 2.19866(-11) \\
  19.0 & 1 &  3.69(-4) & 9.12089(-4) && 1250 & 19 &  2.17(-12) & 2.16638(-11) \\
  19.5 & 1 &  2.96(-4) & 7.46822(-4) && 1300 & 19 &  2.13(-12) & 2.13312(-11) \\
  20 & 1 &  2.37(-4) & 6.10849(-4) &&  1350 & 19 &  2.10(-12) & 2.10036(-11) \\
  21 & 1 &  1.52(-4) & 4.07427(-4) &&  1400 & 19 &  2.07(-12) & 2.06817(-11) \\
  22 & 1 &  9.68(-5) & 2.70724(-4) &&  1450 & 18 &  2.14(-12) & 2.03545(-11) \\
  23 & 1 &  6.17(-5) & 1.79271(-4) &&  1500 & 18 &  2.11(-12) & 2.00263(-11) \\
  24 & 1 &  3.93(-5) & 1.18353(-4) &&  1550 & 18 &  2.07(-12) & 1.97041(-11) \\
  25 & 1 &  2.51(-5) & 7.7946(-5) &&  1600 & 17 &  2.15(-12) & 1.93833(-11) \\
  26 & 1 &  1.61(-5) & 5.12658(-5) &&  1650 & 17 &  2.12(-12) & 1.90539(-11) \\
  27 & 1 &  1.06(-5) & 3.37472(-5) &&  1700 & 17 &  2.08(-12) & 1.87300(-11) \\
  28 & 2 &  3.18(-6) & 2.22937(-5) &&  1750 & 17 &  2.05(-12) & 1.84125(-11) \\
  29 & 2 &  2.00(-6) & 1.45047(-5) &&  1800 & 16 &  2.13(-12)  & 1.80865(-11)\\
  30 & 2 &  1.26(-6) & 9.41621(-6) &&  1850 & 16 &  2.09(-12) & 1.77615(-11) \\
  35 & 2 &  1.22(-7) & 1.05487(-6) &&  1900 & 16 &  2.05(-12) & 1.74427(-11)\\
  40 & 3 &  7.81(-9) & 1.16299(-7) &&  1950 & 15 &  2.14(-12) & 1.71251(-11) \\
  45 & 4 &  5.59(-10) & 1.23376(-8) &&  2000 & 15 &  2.10(-12) & 1.67987(-11) \\
  50 & 7 &  3.44(-11) & 1.30131(-9) &&  2100 & 15 &  2.02(-12) & 1.61646(-11) \\
  75 & 26 &  2.20(-12) & 2.96691(-11) &&  2200 & 14 &  2.07(-12) & 1.55206(-11) \\
  100 & 26 &  2.18(-12) & 2.94551(-11) && 2300 & 13 &  2.13(-12) & 1.48933(-11) \\
  150 & 26 &  2.15(-12) & 2.90681(-11) &&  2400 & 13 &  2.04(-12) & 1.42535(-11) \\
  200 & 26 &  2.13(-12) & 2.87042(-11) && 2500 & 12 &  2.10(-12) & 1.36270(-11) \\
  250 & 25 &  2.18(-12) & 2.83496(-11) && 2600 & 12 &  2.00(-12) & 1.29976(-11) \\
  300 & 25 &  2.15(-12) & 2.79972(-11) && 2700 & 11 &  2.06(-12) & 1.23732(-11) \\
  350 & 25 &  2.13(-12) & 2.76518(-11) && 3000 & 10 &  1.91(-12) & 1.05302(-11) \\
  400 & 25 &  2.10(-12) & 2.73130(-11) && 3200 & 9 &  1.86(-12) & 9.32308(-12) \\
  450 & 24 &  2.16(-12) & 2.69688(-11)  && 3500 & 7 &  1.88(-12) & 7.53751(-12) \\
  500 & 24 &  2.13(-12) & 2.66291(-11)  && 3700 & 6 &  1.83(-12) & 6.39612(-12) \\
  550 & 24 &  2.10(-12) & 2.62963(-11)  && 4000 & 5 &  1.60(-12) & 4.78674(-12) \\
  600 & 23 &  2.16(-12) & 2.59588(-11)  && 4200 & 4 &  1.51(-12) & 3.77702(-12) \\
  650 & 23 &  2.14(-12) & 2.56227(-11)  && 4500 & 3 &  1.23(-12) & 2.46504(-12) \\
  700 & 23 &  2.11(-12) & 2.52897(-11)  && 4700 & 2 &  1.18(-12) & 1.77185(-12) \\
  750 & 22 &  2.17(-12) & 2.49584(-11)  && 5000 & 2 &  6.51(-13) & 9.76476(-13) \\
  800 & 22 &  2.14(-12) & 2.46224(-11)  && 5200 & 2 &  4.38(-13) & 6.56727(-13) \\
  850 & 22 &  2.11(-12) & 2.42914(-11)  && 5500 & 2 &  2.42(-13) & 3.62532(-13) \\
  900 & 22 &  2.08(-12) & 2.39657(-11)  && 5700 & 2 &  1.63(-13) & 2.44112(-13) \\
 \hline
\end{tabular}
\end{table}
For $b=\log(8\cdot10^{11})\approx27.4079$, we have $m=1$, $\delta=9(-6)$ and $\varepsilon=2.84888(-5)$.
For $b=\log10^{16}\approx36.8414$, we have $m=2$, $\delta=5.24(-8)$ and $\varepsilon=4.66629(-7)$.
\begin{table}
\vspace*{-20mm}
\small
  \centering
  \caption{$\eta$ for the case $\varepsilon=C X^{3/4}e^{-X}$}\label{d e}
\begin{tabular}{|l|l|l|l|l|}
  \hline
  $b$&$\eta_1$&$\eta_2$&$\eta_3$&$\eta_4$ \\
 \hline
18.42 & 0.0211674& 0.390116& 7.18984& 132.509 \\
18.43 & 0.0210952& 0.388996& 7.17308& 132.272 \\
18.44 & 0.0210232& 0.387878& 7.15635& 132.035 \\
18.45 & 0.0209968& 0.388441& 7.18616& 132.944 \\
18.5 & 0.0208073& 0.389097& 7.27612& 136.063 \\
18.7 & 0.0195278& 0.371028& 7.04952& 133.941 \\
19 & 0.0177857& 0.346822& 6.76302& 131.879 \\
19.5 & 0.0149364& 0.298729& 5.97458& 119.492 \\
20 & 0.0128278& 0.269384& 5.65707& 118.799 \\
21 & 0.0089634& 0.197195& 4.33829& 95.4423 \\
22 & 0.00622665& 0.143213& 3.2939& 75.7596 \\
23 & 0.00430251& 0.10326& 2.47825& 59.478 \\
24 & 0.00295883& 0.0739707& 1.84927& 46.2317 \\
25 & 0.0020266& 0.0526915& 1.36998& 35.6195 \\
26 & 0.00138418& 0.0373728& 1.00907& 27.2448 \\
27 & 0.000944921& 0.0264578& 0.740818& 20.7429 \\
28 & 0.000646516& 0.018749& 0.54372& 15.7679 \\
29 & 0.00043514& 0.0130542& 0.391626& 11.7488 \\
30 & 0.000329567& 0.0115349& 0.40372& 14.1302 \\
35 & 0.0000421946& 0.00168778& 0.0675114& 2.70046 \\
40 & 5.81494(-6)& 0.000290747& 0.0145374& 0.726868 \\
50 & 9.7598(-8)& 7.31985(-6)& 0.000548989& 0.0411742 \\
75 & 2.96691(-9)& 2.96691(-7)& 0.0000296691& 0.00296691 \\
100 & 4.41826(-9)& 6.6274(-7)& 0.000099411& 0.0149116 \\
150 & 5.81362(-9)& 1.16272(-6)& 0.000232545& 0.046509 \\
200 & 7.17604(-9)& 1.79401(-6)& 0.000448502& 0.112126 \\
250 & 8.50488(-9)& 2.55146(-6)& 0.000765439& 0.229632 \\
300 & 9.79902(-9)& 3.42966(-6)& 0.00120038& 0.420133 \\
350 & 1.10607(-8)& 4.42429(-6)& 0.00176972& 0.707886 \\
400 & 1.22909(-8)& 5.53089(-6)& 0.0024889& 1.12001 \\
450 & 1.34844(-8)& 6.74219(-6)& 0.00337109& 1.68555 \\
500 & 1.4646(-8)& 8.0553(-6)& 0.00443041& 2.43673 \\
550 & 1.57778(-8)& 9.46665(-6)& 0.00567999& 3.408 \\
600 & 1.68732(-8)& 0.0000109676& 0.00712893& 4.6338 \\
650 & 1.79359(-8)& 0.0000125551& 0.00878859& 6.15201 \\
700 & 1.89673(-8)& 0.0000142255& 0.0106691& 8.00183 \\
750 & 1.99667(-8)& 0.0000159734& 0.0127787& 10.223 \\
800 & 2.0929(-8)& 0.0000177897& 0.0151212& 12.853 \\
850 & 2.18623(-8)& 0.0000196761& 0.0177085& 15.9376 \\
900 & 2.27674(-8)& 0.000021629& 0.0205476& 19.5202 \\
 \hline
\end{tabular}
\end{table}

\newpage

\begin{table}
\vspace*{-27mm}
\small
  \centering
  \caption{$\eta$ for the case $\varepsilon=C X^{3/4}e^{-X}$}\label{d e}
\begin{tabular}{|l|l|l|l|l|}
  \hline
  $b$&$\eta_1$&$\eta_2$&$\eta_3$&$\eta_4$ \\
 \hline
950 & 2.36304(-8)& 0.0000236304& 0.0236304& 23.6304 \\
1000 & 2.44643(-8)& 0.0000256875& 0.0269719& 28.3204 \\
1050 & 2.52703(-8)& 0.0000277973& 0.030577& 33.6347 \\
1100 & 2.60413(-8)& 0.0000299475& 0.0344396& 39.6055 \\
1150 & 2.67764(-8)& 0.0000321317& 0.038558& 46.2696 \\
1200 & 2.74833(-8)& 0.0000343541& 0.0429427& 53.6783 \\
1250 & 2.81629(-8)& 0.0000366118& 0.0475953& 61.8739 \\
1300 & 2.87972(-8)& 0.0000388762& 0.0524828& 70.8518 \\
1350 & 2.94051(-8)& 0.0000411671& 0.057634& 80.6876 \\
1400 & 2.99885(-8)& 0.0000434834& 0.0630509& 91.4238 \\
1450 & 3.05317(-8)& 0.0000457975& 0.0686963& 103.044 \\
1500 & 3.10407(-8)& 0.0000481131& 0.0745754& 115.592 \\
1550 & 3.15265(-8)& 0.0000504425& 0.080708& 129.133 \\
1600 & 3.19825(-8)& 0.0000527712& 0.0870724& 143.669 \\
1650 & 3.23916(-8)& 0.0000550657& 0.0936117& 159.14 \\
1700 & 3.27774(-8)& 0.0000573605& 0.100381& 175.667 \\
1750 & 3.31425(-8)& 0.0000596565& 0.107382& 193.287 \\
1800 & 3.346(-8)& 0.0000619011& 0.114517& 211.856 \\
1850 & 3.37469(-8)& 0.000064119& 0.121826& 231.47 \\
1900 & 3.40132(-8)& 0.0000663258& 0.129335& 252.204 \\
1950 & 3.42502(-8)& 0.0000685004& 0.137001& 274.002 \\
2000 & 3.52772(-8)& 0.0000740822& 0.155573& 326.703 \\
2100 & 3.55621(-8)& 0.0000782365& 0.17212& 378.665 \\
2200 & 3.56974(-8)& 0.0000821041& 0.188839& 434.331 \\
2300 & 3.57438(-8)& 0.0000857851& 0.205884& 494.122 \\
2400 & 3.56337(-8)& 0.0000890842& 0.22271& 556.776 \\
2500 & 3.54302(-8)& 0.0000921186& 0.239508& 622.722 \\
2600 & 3.50936(-8)& 0.0000947527& 0.255832& 690.747 \\
2700 & 3.71195(-8)& 0.000111358& 0.334075& 1002.23 \\
3000 & 3.36966(-8)& 0.000107829& 0.345053& 1104.17 \\
3200 & 3.26308(-8)& 0.000114208& 0.399727& 1399.04 \\
3500 & 2.78888(-8)& 0.000103189& 0.381798& 1412.65 \\
3700 & 2.55845(-8)& 0.000102338& 0.409352& 1637.41 \\
4000 & 2.01043(-8)& 0.0000844381& 0.35464& 1489.49 \\
4200 & 1.69966(-8)& 0.0000764846& 0.344181& 1548.81 \\
4500 & 1.15857(-8)& 0.0000544527& 0.255928& 1202.86 \\
4700 & 8.85926(-9)& 0.0000442963& 0.221481& 1107.41 \\
5000 & 5.07767(-9)& 0.0000264039& 0.1373& 713.962 \\
5200 & 3.612(-9)& 0.000019866& 0.109263& 600.946 \\
5500 & 2.06644(-9)& 0.0000117787& 0.0671385& 382.689 \\
5700 & 1.40731(-9)& 8.11312(-6)& 0.0467721& 269.641 \\
 \hline
\end{tabular}
\end{table}
лллллллллллллллллллллллллллл

For $b\geq5213$,
$$
\varepsilon<\Omega_1^\ast e^{-b/2}+\Omega_2^\ast+m\delta/2+\log(2\pi)e^{-b}
$$
so that instead of $\Omega^\ast$ we use general theorem $\varepsilon=\sqrt{8/\pi}....)$.
ллллллллллллллллллллллллллллллллл
D=2500, minimize all:
\newpage
\begin{table}
\vspace*{-20mm}
\small
  \centering
  \caption{$|\psi(x)-x|<x\varepsilon,\quad(x\geq e^b)$, $\varepsilon=C' X^{1/2}e^{-X}$}\label{d e}
\begin{tabular}{|c|c|c|c|c|c|c|c|c|}
  \hline
  $b$&$m$&$\delta$&$\varepsilon$&&$b$&$m$&$\delta$&$\varepsilon$ \\
 \hline 
  18.42 & 1 & 4.78(-4) & 1.14790(-3) && 900  & 22 &  2.08(-12) & 2.39881(-11) \\
  18.43 & 1 & 4.76(-4) & 1.14336(-3) && 950  & 21 &  2.15(-12) & 2.36469(-11) \\
  18.44 & 1 & 4.74(-4) & 1.13884(-3) && 1000 & 21 &  2.12(-12) & 2.33114(-11) \\
  18.45 & 1 & 4.71(-4) & 1.13434(-3) && 1050 & 21 &  2.09(-12) & 2.29819(-11) \\
  18.5 & 1 &  4.61(-4) & 1.11208(-3) && 1100 & 20 &  2.16(-12) & 2.26511(-11) \\
  18.7 & 1 &  4.22(-4) & 1.02723(-3) && 1150 & 20 &  2.13(-12) & 2.23185(-11) \\
  19.0 & 1 &  3.70(-4) & 9.11615(-4) && 1200 & 20 &  2.09(-12) & 2.19902(-11) \\
  19.5 & 1 &  2.96(-4) & 7.46453(-4) && 1250 & 19 &  2.17(-12) & 2.16664(-11) \\
  20   & 1 &  2.37(-4) & 6.10561(-4) && 1300 & 19 &  2.13(-12) & 2.13331(-11) \\
  21   & 1 &  1.52(-4) & 4.07253(-4) && 1350 & 19 &  2.10(-12) & 2.10050(-11) \\
  22 & 1   &  9.68(-5) & 2.70618(-4) && 1400 & 19 &  2.07(-12) & 2.06828(-11) \\
  23 & 1   &  6.17(-5) & 1.79207(-4) && 1450 & 18 &  2.14(-12) & 2.03552(-11) \\
  24 & 1   &  3.93(-5) & 1.18314(-4) && 1500 & 18 &  2.11(-12) & 2.00268(-11) \\
  25 & 1   &  2.51(-5) & 7.79224(-5) && 1550 & 18 &  2.07(-12) & 1.97045(-11) \\
  26 & 1   &  1.61(-5) & 5.12515(-5) && 1600 & 17 &  2.15(-12) & 1.93836(-11) \\
  27 & 1   &  1.06(-5) & 3.37385(-5) && 1650 & 17 &  2.12(-12) & 1.90541(-11) \\
  28 & 1   &  7.22(-6) & 2.23274(-5) && 1700 & 17 &  2.08(-12) & 1.87301(-11) \\
  29 & 1   &  5.26(-6) & 1.49727(-5) && 1750 & 17 &  2.05(-12) & 1.84126(-11)\\
  30 & 2   &  1.26(-6) & 9.41428(-6) && 1800 & 16 &  2.13(-12) & 1.80866(-11) \\
  35 & 2   &  1.22(-7) & 1.05471(-6) && 1850 & 16 &  2.09(-12) & 1.77616(-11)\\
  40 & 3   &  7.81(-9) & 1.16290(-7) && 1900 & 16 &  2.05(-12) & 1.74427(-11) \\
  45 & 4   &  5.60(-10) & 1.23408(-8) &&1950 & 15 &  2.14(-12) & 1.71251(-11)  \\
  50 & 7   &  3.45(-11) & 1.30541(-9) &&2000 & 15 &  2.10(-12) & 1.67987(-11)  \\
  75 & 26   &  2.20(-12) & 3.32667(-11) &&2100 & 15 &  2.02(-12) & 1.61646(-11) \\
  100 & 26  &  2.18(-12) & 3.25398(-11) &&2200 & 14 &  2.07(-12) & 1.55206(-11) \\
  150 & 26  &  2.16(-12) & 3.13387(-11) &&2300 & 13 &  2.12(-12) & 1.48944(-11) \\
  200 & 26  &  2.13(-12) & 3.03713(-11) &&2400 & 13 &  2.04(-12) & 1.42535(-11) \\
  250 & 25  &  2.18(-12) & 2.95752(-11)&&2500 & 12 &  2.10(-12) & 1.36270(-11) \\
  300 & 25  &  2.15(-12) & 2.88982(-11) &&2600 & 12 &  2.00(-12) & 1.29976(-11) \\
  350 & 25  &  2.13(-12) & 2.83142(-11)  &&2700 & 11 &  2.06(-12) & 1.23732(-11) \\
  400 & 25  &  2.10(-12) & 2.78000(-11)  &&3000 & 10 &  1.92(-12) & 1.05303(-11) \\
  450 & 24  &  2.16(-12) & 2.73267(-11) &&3200 & 9 &  1.86(-12) & 9.32308(-12) \\
  500 & 24  &  2.13(-12) & 2.68923(-11) &&3500 & 7 &  1.89(-12) & 7.53761(-12) \\
  550 & 24  &  2.10(-12) & 2.64897(-11)  &&3700 & 6 &  1.83(-12) & 6.39612(-12) \\
  600 & 23  &  2.16(-12) & 2.61010(-11)  &&4000 & 5 &  1.60(-12) & 4.78674(-12) \\
  650 & 23  &  2.14(-12) & 2.57273(-11)  &&4500 & 3 &  1.23(-12) & 2.46504(-12) \\
  700 & 23  &  2.11(-12) & 2.53666(-11)  &&4700 & 2 &  1.20(-12) & 1.77229(-12) \\
  750 & 22  &  2.17(-12) & 2.50149(-11) &&5000 & 2 &  6.51(-13) & 9.76476(-13) \\
  800 & 22  &  2.14(-12) & 2.46639(-11)  &&5100 & 2 &  5.34(-13) & 8.00754(-13) \\
  850 & 22 &  2.11(-12) & 2.43220(-11) &&5200 & 2 &  4.38(-13) & 6.56727(-13) \\
  \hline
\end{tabular}
\end{table}

\newpage

\begin{table}
\vspace*{-20mm}
\small
  \centering
  \caption{$\eta$ for the case $\varepsilon=C' X^{1/2}e^{-X}$}\label{eta}
\begin{tabular}{|c|c|c|c|c|}
  \hline
$b$&$\eta_1$&$\eta_2$&$\eta_3$&$\eta_4$\\
\hline
  18.42 & 0.0211558& 0.389901& 7.18587& 132.436 \\
  18.43 & 0.0210836& 0.388781& 7.16912& 132.199 \\
  18.44 & 0.0210116& 0.387664& 7.15241& 131.962 \\
  18.45 & 0.0209853& 0.388227& 7.18220& 132.871 \\
  18.5  & 0.0207960& 0.388884& 7.27214& 135.989 \\
  18.7  & 0.0195173& 0.370829& 7.04574& 133.869 \\
  19.0  & 0.0177765& 0.346641& 6.75951& 131.81 \\
  19.5  & 0.0149291& 0.298581& 5.97162& 119.432 \\
  20    & 0.0128218& 0.269258& 5.65441& 118.743 \\
  21    & 0.00895956& 0.19711& 4.33643& 95.4014 \\
  22    & 0.00622421& 0.143157& 3.29261& 75.73 \\
  23    & 0.00430097& 0.103223& 2.47736& 59.4567 \\
  24    & 0.00295785& 0.0739463& 1.84866& 46.2165 \\
  25    & 0.00202598& 0.0526755& 1.36956& 35.6087 \\
  26    & 0.00138379& 0.0373624& 1.00878& 27.2372 \\
  27    & 0.000944678& 0.026451& 0.740628& 20.7376 \\
  28    & 0.000647495& 0.0187774& 0.544543& 15.7918 \\
  29    & 0.000449182& 0.0134755& 0.404264& 12.1279 \\
  30    & 0.0003295& 0.0115325& 0.403637& 14.1273 \\
  35    & 0.0000421884& 0.00168754& 0.0675015& 2.70006 \\
  40    & 5.23306(-6)& 0.000235488& 0.0105969& 0.476863 \\
  45    & 6.17042(-7)& 0.0000308521& 0.00154261& 0.0771303\\
  50    & 9.79061(-8)&7.34296(-6)&0.000550722&0.0413041 \\
  75    & 3.32667(-9)& 3.32667(-7)& 0.0000332667& 0.00332667 \\
  100   & 4.88096(-9)& 7.32145(-7)& 0.000109822& 0.0164733 \\
  150   & 6.26774(-9)& 1.25355(-6)& 0.00025071& 0.0501419 \\
  200   & 7.59281(-9)& 1.8982(-6)& 0.000474551& 0.118638 \\
  250   & 8.87255(-9)& 2.66177(-6)& 0.00079853& 0.239559 \\
  300   & 1.01144(-8)& 3.54003(-6)& 0.00123901& 0.433654 \\
  350   & 1.13257(-8)& 4.53027(-6)& 0.00181211& 0.724843 \\
  400   & 1.251(-8)& 5.62949(-6)& 0.00253327& 1.13997 \\
  450   & 1.36634(-8)& 6.83168(-6)& 0.00341584& 1.70792 \\
  500   & 1.47907(-8)& 8.13491(-6)& 0.0044742& 2.46081 \\
  550   & 1.58938(-8)& 9.5363(-6)& 0.00572178& 3.43307 \\
  600   & 1.69656(-8)& 0.0000110277& 0.00716799& 4.65919 \\
  650   & 1.80091(-8)& 0.0000126064& 0.00882445& 6.17712 \\
  700   & 1.90249(-8)& 0.0000142687& 0.0107015& 8.02615 \\
  750   & 2.0012(-8)& 0.0000160096& 0.0128076& 10.2461 \\
  800   & 2.09643(-8)& 0.0000178197& 0.0151467& 12.8747 \\
  850   & 2.18898(-8)& 0.0000197008& 0.0177307& 15.9577 \\
  \hline
\end{tabular}
\end{table}
\newpage

\begin{table}
\vspace*{-27mm}
\small
  \centering
  \caption{$\eta$ for the case $\varepsilon=C' X^{1/2}e^{-X}$}\label{eta}
\begin{tabular}{|c|c|c|c|c|}
  \hline
$b$&$\eta_1$&$\eta_2$&$\eta_3$&$\eta_4$\\
\hline
900  & 2.27887(-8)& 0.0000216493& 0.0205668& 19.5385 \\
950  & 2.36469(-8)& 0.0000236469& 0.0236469& 23.6469 \\
1000 & 2.4477(-8)& 0.0000257009& 0.0269859& 28.3352 \\
1050 & 2.52801(-8)& 0.0000278081& 0.0305889& 33.6478 \\
1100 & 2.60488(-8)& 0.0000299561& 0.0344495& 39.617 \\
1150 & 2.67822(-8)& 0.0000321386& 0.0385663& 46.2796 \\
1200 & 2.74877(-8)& 0.0000343597& 0.0429496& 53.687 \\
1250 & 2.81663(-8)& 0.0000366162& 0.0476011& 61.8814 \\
1300 & 2.87997(-8)& 0.0000388796& 0.0524875& 70.8582 \\
1350 & 2.94071(-8)& 0.0000411699& 0.0576378& 80.693 \\
1400 & 2.999(-8)& 0.0000434855& 0.063054& 91.4283 \\
1450 & 3.05328(-8)& 0.0000457993& 0.0686989& 103.048 \\
1500 & 3.10416(-8)& 0.0000481145& 0.0745774& 115.595 \\
1550 & 3.15272(-8)& 0.0000504435& 0.0807096& 129.135 \\
1600 & 3.1983(-8)& 0.000052772& 0.0870738& 143.672 \\
1650 & 3.2392(-8)& 0.0000550663& 0.0936128& 159.142 \\
1700 & 3.27777(-8)& 0.000057361& 0.100382& 175.668 \\
1750 & 3.31427(-8)& 0.0000596569& 0.107382& 193.288 \\
1800 & 3.34602(-8)& 0.0000619014& 0.114518& 211.857 \\
1850 & 3.3747(-8)& 0.0000641193& 0.121827& 231.471\\
1900 & 3.40133(-8)& 0.0000663259& 0.129336& 252.204\\
1950 & 3.42503(-8)& 0.0000685005& 0.137001& 274.002 \\
2000 & 3.52773(-8)& 0.0000740823& 0.155573& 326.703 \\
2100 & 3.55621(-8)& 0.0000782366& 0.172121& 378.665 \\
2200 & 3.56974(-8)& 0.0000821041& 0.188839& 434.331 \\
2300 & 3.57465(-8)& 0.0000857917& 0.2059& 494.16 \\
2400 & 3.56337(-8)& 0.0000890842& 0.22271& 556.776 \\
2500 & 3.54303(-8)& 0.0000921187& 0.239509& 622.722 \\
2600 & 3.50936(-8)& 0.0000947527& 0.255832& 690.747 \\
2700 & 3.71195(-8)& 0.000111358& 0.334075& 1002.23 \\
3000 & 3.3697(-8)& 0.00010783& 0.345057& 1104.18 \\
3200 & 3.26308(-8)& 0.000114208& 0.399727& 1399.04 \\
3500 & 2.8266(-8)& 0.000105998& 0.397491& 1490.59 \\
3750 & 2.4491(-8)& 0.0000979639& 0.391855& 1567.42 \\
4000 & 2.01043(-8)& 0.0000844381& 0.35464& 1489.49 \\
4200 & 1.69978(-8)& 0.0000764902& 0.344206& 1548.93 \\
4500 & 1.15857(-8)& 0.0000544527& 0.255928& 1202.86 \\
4700 & 8.86144(-9)& 0.0000443072& 0.221536& 1107.68 \\
5000 & 4.98003(-9)& 0.0000253981& 0.12953& 660.606 \\
5100 & 4.16392(-9)& 0.0000216524& 0.112592& 585.48 \\
5200 & 3.42352(-9)& 0.0000178468& 0.0930354& 484.993 \\
\hline
\end{tabular}
\end{table}

\newpage

\begin{table}
\vspace*{-20mm}
\small
  \centering
  \caption{$\eta$ for the case $\varepsilon=C' X^{1/2}e^{-X}$}\label{eta}
\begin{tabular}{|l|l|l|l|l|l|l|l|}
  \hline
$b$&$m$&$\delta$&$\varepsilon$&$\eta_1$&$\eta_2$&$\eta_3$&$\eta_4$\\
\hline
3800 &6&1.67(-12)&5.86122(-12)& 2.23312(-8)& 0.000085082& 0.324163& 1235.06  \\
3810 &5&1.94(-12)&5.80739(-12)& 2.21842(-8)& 0.0000847437& 0.323721& 1236.61  \\
3820 &5&1.92(-12)&5.74859(-12)& 2.20171(-8)& 0.0000843255& 0.322967& 1236.96  \\
3830 &5&1.90(-12)&5.69039(-12)& 2.18511(-8)& 0.0000839082& 0.322207& 1237.28  \\
3840 &5&1.88(-12)&5.63277(-12)& 2.16862(-8)& 0.0000834918& 0.321443& 1237.56  \\
3850 &5&1.86(-12)&5.57575(-12)& 2.15224(-8)& 0.0000830764& 0.320675& 1237.8  \\
3860 &5&1.84(-12)&5.51930(-12)& 2.13597(-8)& 0.000082662& 0.319902& 1238.02  \\
3870 &5&1.82(-12)&5.46344(-12)& 2.11982(-8)& 0.0000822488& 0.319126& 1238.21  \\
3880 &5&1.80(-12)&5.40817(-12)& 2.10378(-8)& 0.0000818369& 0.318346& 1238.36  \\
3890 &5&1.78(-12)&5.35348(-12)& 2.08786(-8)& 0.0000814265& 0.317563& 1238.5  \\
3900 &5&1.77(-12)&5.29927(-12)& 2.07201(-8)& 0.0000810158& 0.316772& 1238.58  \\
3910 &5&1.75(-12)&5.24559(-12)& 2.05627(-8)& 0.0000806058& 0.315975& 1238.62  \\
3920 &5&1.73(-12)&5.19249(-12)& 2.04065(-8)& 0.0000801975& 0.315176& 1238.64  \\
3930 &5&1.71(-12)&5.13998(-12)& 2.02515(-8)& 0.000079791& 0.314376& 1238.64  \\
3940 &5&1.70(-12)&5.08798(-12)& 2.00975(-8)& 0.0000793852& 0.313572& 1238.61  \\
3950 &5&1.68(-12)&5.03642(-12)& 1.99442(-8)& 0.0000789791& 0.312757& 1238.52  \\
3960 &5&1.66(-12)&4.98545(-12)& 1.97922(-8)& 0.0000785752& 0.311944& 1238.42  \\
3970 &5&1.64(-12)&4.93509(-12)& 1.96417(-8)& 0.0000781739& 0.311132& 1238.31  \\
3980 &5&1.63(-12)&4.88504(-12)& 1.94913(-8)& 0.0000777704& 0.310304& 1238.11  \\
3990 &5&1.61(-12)&4.83561(-12)& 1.93424(-8)& 0.0000773697& 0.309479& 1237.92  \\
4000 &5&1.60(-12)&4.78674(-12)& 1.91948(-8)& 0.0000769713& 0.308655& 1237.71  \\
\hline
\end{tabular}
\end{table}

\end{document}